\documentclass[12pt]{article}

\usepackage{mathscinet}
\usepackage[margin=1in]{geometry}

\usepackage{amsmath}
\usepackage{amssymb}
\usepackage{amsthm}
\usepackage{url}
\usepackage{MnSymbol}
\usepackage{graphicx}
\usepackage{todonotes}
\usepackage{rotating}

\usepackage{thm-restate}

\usepackage{enumitem}

\usepackage{tikz} 
\usetikzlibrary{automata,arrows,shapes,topaths,trees,backgrounds,decorations.pathreplacing,positioning}

\usepackage{hyperref}

\usepackage[all,cmtip]{xy}
\usepackage{varwidth}

\newtheorem{theorem}{Theorem}[section]

\newtheorem{lemma}[theorem]{Lemma}

\newtheorem{proposition}[theorem]{Proposition}
\newcounter{claims}[theorem]

\theoremstyle{definition}
\newtheorem{definition}[theorem]{Definition}

\numberwithin{subcase}{cases}

\theoremstyle{remark}

\newtheorem{claim}[claims]{Claim}

\newcommand{\mc}[1]{\mathcal{#1}}

\newcommand{\lkill}{\ell_{\mathsf{kill}}}

\newcommand{\Inv}[2]{{#2}^{(-#1)}}

\DeclareMathOperator{\Mod}{Mod}

\newcommand{\bfPi}{\boldsymbol{\Pi}}

\newcommand{\bfSigma}{\boldsymbol{\Sigma}}

\begin{document}
	
	\title{Unfriendly jump inversions}
	\author{Matthew Harrison-Trainor\thanks{The author was supported by a Sloan Research Fellowship and by the National Science Foundation under Grants No.\ \mbox{DMS-2419591} and \mbox{DMS-2452105}. A preprint of this article was circulated under the name \textit{Relative to any non-arithmetic set}; in addition to the change in title, the introduction has been rewritten. The research was completed and the original manusript written without the use of AI tools. In the latest revision, ChatGPT 5.6 Sol was used to search for both linguistic and mathematical errors.}}
	
	\makeatletter
	\def\blfootnote{\xdef\@thefnmark{}\@footnotetext}
	\makeatother
	
	\maketitle
	
	\begin{abstract}
		We introduce a new technique in computable structure theory which we call unfriendly jump inversions. In contrast to the standard technique of jump inversions, we obtain the maximal possible difference between the effective and non-effective settings, namely, the $n$th unfriendly jump inversion behaves like an $n$th jump inversion from the non-effective perspective, but like a $2n$th jump inversion from the effective perspective. We give several applications, among them the construction of a structure which has no arithmetic copy but which has a copy computable from every non-arithmetic set. This has been an open question in the study of degree spectra for some time. We also show that for every $n \geq 2$ there is a computable structure with a $\Pi_n$ Scott sentence but no computable $\Sigma_{2n}$ Scott sentence.
	\end{abstract}
	
	\section{Introduction}
	
	This paper contributes to the program of effective or algorithmic mathematics whose aim is to understand the algorithmic content of mathematics. This program began in the early 1900s with work by, e.g., Dehn \cite{Dehn1911}, Grete Hermann \cite{Hermann1926}, and van der Waerden \cite{vanDerWaerden1930} on algorithmic aspects of algebra such as the word problem for groups or effective procedures in ring theory and field theory. Later, the development of computability theory led to a formal understanding of what an algorithm is. Starting in the 1950s, Fr\"ohlich and Shepherdson \cite{FroehlichShepherdson1956}, Rabin \cite{Rabin1960}, and Mal'cev \cite{Malcev1961,Malcev1962} were able to study effective mathematics with full precision and calibrate the computational difficulty of algorithmic questions. In particular, one could now show that certain questions are algorithmically unsolvable, as Novikov \cite{Novikov1955} and Boone \cite{Boone1959} did with the word problem for finitely presented groups. Out of this program has grown the thriving field of computable structure theory in which one attempts to measure the complexity appearing in mathematics using computational tools. See, for example, the recent books by Montalb\'an \cite{MonBook1,MonBook2} and Downey and Melnikov \cite{DowneyMelnikov2026}.
	
	In computable structure theory one studies countable mathematical objects, such as graphs, groups, or linear orders, which we formally call structures and which consist of a (countable) domain together with functions, relations, and distinguished constant elements. We ask questions like: How hard is it to characterize a structure up to isomorphism?  How hard is it to decide whether two structures are isomorphic, or to construct an isomorphism between them? How hard is it to build a copy of a structure? These questions can be stated and answered formally in the language of computability theory, that is, in terms of algorithms, with complexity measured in the Turing degrees.
	
	There are two other areas of mathematical logic in which we also consider such questions, namely invariant descriptive set theory and the model theory of infinitary logic. Each of these has its own language. In descriptive set theory, we use the language of Borel sets and Borel-measurable functions with their complexity measured in the Borel hierarchy. In infinitary model theory, we use definability in the infinitary logic $\mc{L}_{\omega_1 \omega}$ with complexity measured in terms of alternations of quantifiers. There are strong connections especially between computable structure theory, infinitary model theory, and descriptive set theory. Though each of these fields has its own language there are bridges between them such as the Lopez-Escobar theorem \cite{Lop66,Vaught,vandenBoom} and the fact that a function is continuous if and only if it is computable relative to some oracle.
	
	In general there are two settings that one can work in. The first is the \textit{effective setting} of computable structure theory, which also encompasses effective descriptive set theory (in which one considers Borel sets that can be constructed algorithmically) and computable infinitary logic (in which one considers formulas which can be constructed algorithmically). The \textit{non-effective setting} is the natural setting of descriptive set theory and of infinitary logic where one considers all Borel sets and all formulas. In computable structure theory, this often corresponds to considering algorithms which have access to an arbitrary oracle. These two settings are intimately connected and even when working in one setting, it is often useful to pass into the other setting.
	
	In this paper we introduce a new technique which we call \emph{unfriendly jump inversions}. This name is in reference to a standard technique of jump inversions which produces from one structure another structure of increased complexity. The $n$th jump inversion of a structure will be more complex by an additive factor of $n$ in all measures of complexity, including, for example, its Scott rank which measures the complexity of describing  the structure up to isomorphism. From the point of view of infinitary logic, jump inversions are the opposite of a definitional expansion: a relation which was originally in the language is removed from the language and becomes only a definable relation. From the point of view of descriptive set theory, a jump inversion can be thought of as a change of topology on $\Mod(\mc{L})$ which, in contrast to most uses of change of topology, makes fewer open sets rather than more. 
	
	The usual jump inversions have the same increase in complexity in both the effective and non-effective settings. This is related to a formal notion of friendliness which appears in their definition; in the remainder of the paper, we will call these friendly jump inversions. Our unfriendly jump inversions, in contrast, create the largest possible difference between the effective and non-effective settings. In general, when taking the $n$th unfriendly jump inversion, in the non-effective setting the complexity still increases by $n$, but in the effective setting the complexity increases by $2n$. This seemingly innocuous difference turns out to have deep consequences. The author developed unfriendly jump inversions with the following two applications in mind:
	\begin{enumerate}
		\item There is a structure which does not have an arithmetic copy, but which has a copy computable from every non-arithmetic set. This answers a well-known and long-standing question in the study of degree spectra. Slaman \cite{Slaman} and Wehner \cite{Wehner} had independently shown that there is a structure which does not have a computable copy but which does have a copy computable from every non-computable set, and Greenberg, Montalb\'an, and Slaman \cite{nonhyp} later did the same for the hyperarithmetic degrees. It was open for some time until now whether one could do this for the arithmetic degrees. See Section \ref{sec:sw} for the history of this question and other discussion.
		
		\item For each $n \geq 2$ there is a computable structure with a $\Pi_{n}$ Scott sentence (hence a computable $\Pi_{2n}$ Scott sentence) but no computable $\Sigma_{2n}$ Scott sentence. This answers a question raised in \cite{AlvirCsimaHT} where it was noted that standard friendly jump inversions only allow one to prove that for each $n$ there is a computable structure with a $\Pi_{n}$ Scott sentence but no computable $\Sigma_{n+2}$ Scott sentence.
	\end{enumerate}
	These applications and others are discussed in more detail in the second part of this introduction and proved later in the paper.

	We believe that unfriendly jump inversions, like the standard friendly jump inversions, will become a standard tool of computable structure theory. Already several papers have been written following up on our work in this paper including several other applications:
	\begin{enumerate}
		\item In this paper we treat only finite levels because that was what was required to show that the non-arithmetic degrees are a degree spectrum. Andrews, Gonzalez, and Zhu \cite{AndrewsGonzalezZhu2026} have already extended unfriendly jump inversions through the transfinite and extended many of the applications of this paper.
		\item Gonzalez and Knight \cite{GonzalezKnight2026} show that the Friedman-Stanley embedding preserves the existence and complexity of computable infinitary Scott sentences. The image of the Friedman-Stanley embedding is not Borel \cite{Gao}. Gonzalez and Knight show that the restricted image of the structures of Scott rank at most $\alpha$ is (lightface) $\Pi^0_{2\alpha+2}$-complete relative to a presentation of $\alpha$. They use unfriendly jump inversions for the hardness result.
		\item Gonzalez, Harrison-Trainor, and Knight \cite{GonzalezHarrisonTrainorKnight2026} show that there is an effective $\Delta^0_{\alpha+1}$ operator that reverses the Friedman-Stanley embedding for structures of Scott rank at most $\alpha$. They use unfriendly jump inversions to show that there is no such $\Delta^0_n$ operator for finite $n$.
		\item Franklin, Rossegger, and Turetsky \cite{FranklinRosseggerTuretsky2026} use unfriendly jump inversions to show that, for every $\alpha < \omega_1^{\mathrm{CK}}$, a Turing degree $\mathbf{d}$ is high for Scott rank $\alpha$ if and only if $\mathbf{d}$ computes every $\Delta^0_{2\alpha+1}$ set.
	\end{enumerate}
	We expect more applications are yet to come.
	
	The organization of this paper is as follows. In the remainder of the introduction, we begin by giving some background in Section \ref{sec:intro-back-and-forth}, a short exposition of the standard technique of friendly jump inversions in Section \ref{sec:intro-jump-inversions}, and then a description of our unfriendly jump inversions in Section \ref{sec:intro-unfriendly-jump-inversions}. In Section \ref{sec:intro-applications} we outline the applications of unfriendly jump inversions we obtain in this paper including that the non-arithmetic degrees are a degree spectrum in Section \ref{sec:sw}. In Section \ref{sec:computing-optimal-sentences} we prove some preliminary results about computing optimal Scott sentences, and then we develop unfriendly jump inversions in full detail in Section \ref{sec:unfriendly-jump-inversions}. Section \ref{sec:applications} contains most of the applications, with Section \ref{sec:non-arithmetic-proof} containing the argument that the non-arithmetic degrees are a degree spectrum.
	
	The reader who is interested only in the fact that the non-arithmetic degrees are a degree spectrum can skip ahead to Section \ref{sec:sw} for an introduction to the problem which can be read independently from the rest of the introduction, take Theorem \ref{thm:jump-inversion-structures} for granted, and then skip ahead to Section \ref{sec:non-arithmetic-proof}.
	
	\subsection{Back-and-forth relations and friendliness}\label{sec:intro-back-and-forth}
	
	One of the main themes in computable structure theory over the last thirty years, especially championed by Ash and Knight in their book \cite{AshKnight} and more recently Montalb\'an in his books \cite{MonBook1,MonBook2}, has been the connection between computational properties and structural or definable properties of structures. The logic we use is the infinitary logic $\mc{L}_{\omega_1 \omega}$ which allows countably infinite conjunctions and disjunctions. The complexity of the formulas defining a property is particularly important, where by complexity we mean the number of alternations between existential and universal quantifiers.
	\begin{itemize}
		\item A formula is $\Sigma_0$ and $\Pi_0$ if it is finitary quantifier-free.
		\item A formula is $\Sigma_{n+1}$ if it is of the form
		\[ \bigdoublevee_i \exists \bar{x}_i \psi_i(\bar{x}_i)\]
		where each $\psi_i$ is $\Pi_{n}$.
		\item A formula is $\Pi_{n+1}$ if it is of the form
		\[ \bigdoublewedge_i \forall \bar{x}_i \psi_i(\bar{x}_i)\]
		where each $\psi_i$ is $\Sigma_{n}$.
	\end{itemize}
	\noindent (This complexity hierarchy continues through the infinite ordinals but in this paper we will mostly consider the case of finite $n$. This is both for simplicity and because all of the phenomena that we are interested in appear there.)
	
	The intuition is that whenever a structure has some computational property there is a structural reason as expressed in infinitary logic at an appropriate level of complexity. While this is a useful guiding principle it is often only strictly true when we consider all copies of a structure and not just computable copies. For example, consider computable categoricity:
	\begin{enumerate}
		\item A computable structure $\mc{A}$ is \textit{computably categorical} if every computable copy $\mc{B}$ of $\mc{A}$ is isomorphic to $\mc{A}$ via an isomorphism that is computable.
		\item A computable structure $\mc{A}$ is \textit{relatively computably categorical} if every copy $\mc{B}$ of $\mc{A}$ is isomorphic to $\mc{A}$ via an isomorphism that is computable relative to $\mc{B}$.
		\item A structure $\mc{A}$ is \textit{computably categorical on a cone} if there is $X$ such that for all $Y \geq_T X$, every $Y$-computable copy $\mc{B}$ of $\mc{A}$ is isomorphic to $\mc{A}$ via an isomorphism that is $Y$-computable.
	\end{enumerate}
	Relative computable categoricity and computable categoricity on a cone admit characterizations in terms of syntactic definability:  $\mc{A}$ is computably categorical on a cone if and only if the automorphism orbits of $\mc{A}$ are definable by existential formulas, and $\mc{A}$ is relatively computably categorical if and only if the automorphism orbits of $\mc{A}$ are definable by existential formulas and there is a c.e.\ set of such formulas \cite{AshKnightManasseSlaman,Chisholm}. On the other hand there are computable structures which are computably categorical but not relatively computably categorical \cite{Goncharov77}, and computable categoricity does not admit even a hyperarithmetic characterization \cite{CompCat}. The situation is similar more generally for $\Delta^0_\alpha$-categoricity where we ask for $\Delta^0_\alpha$ isomorphisms.

	However, under certain effectivity conditions on the structures we are considering we can often recover some reasonable characterization. Goncharov \cite{Goncharov75} showed that if $\mc{A}$ is a computable structure whose finitary $\forall \exists$ diagram is computable, then $\mc{A}$ is computably categorical if and only if $\mc{A}$ is relatively computably categorical if and only if there is a c.e.\ collection of existential formulas defining the automorphism orbits of $\mc{A}$. For $\Delta^0_\alpha$-categoricity we have similar theorems of Ash \cite{Ash87}, where in the effectiveness conditions the key is to understand the back-and-forth relations of the structures.
	\begin{definition}
		We define the standard asymmetric back-and-forth relations $(\mc{A},\bar{a}) \leq_n (\mc{B},\bar{b})$ inductively as follows.
		\begin{enumerate}
			\item $(\mc{A},\bar{a}) \leq_0 (\mc{B},\bar{b})$ if and only if $\bar{a}$ and $\bar{b}$ satisfy the same atomic and negated atomic formulas (using the first $|\bar{a}|$-many symbols).
			\item $(\mc{A},\bar{a}) \leq_n (\mc{B},\bar{b})$ if for every $\bar{b}' \in \mc{B}$ there is $\bar{a}' \in \mc{A}$ such that $(\mc{B},\bar{b}\bar{b}') \leq_{n-1} (\mc{A},\bar{a}\bar{a}')$.
		\end{enumerate}
		We write $\mc{A} \leq_n \mc{B}$ if $(\mc{A},\varnothing) \leq_n (\mc{B},\varnothing)$. We also write $\mc{A} \equiv_n \mc{B}$ if $\mc{A} \leq_n \mc{B}$ and $\mc{B} \leq_n \mc{A}$.
	\end{definition}
	
	\noindent This particular definition of the back-and-forth relations is intimately connected with the infinitary logic $\mc{L}_{\omega_1 \omega}$: $\mc{A} \leq_n \mc{B}$ if and only if every $\Pi_n$ sentence true of $\mc{A}$ is true of $\mc{B}$, and also if and only if every $\Sigma_n$ sentence true of $\mc{B}$ is true of $\mc{A}$ \cite{Karp}.
	
	A structure or class of structures is said by Ash and Knight \cite{AshKnight} to be $n$-friendly if the $n$-back-and-forth relations are c.e.  Montalb\'an in his recent book \cite{MonBook2} requires instead that the back-and-forth relations be computable (but does not use the term ``friendly''). Taking poetic license, we will adopt the terminology of Ash and Knight but take the stronger definition from Montalb\'an (a decision which is of no real consequence as we will use friendliness only in a motivational role and not in our theorems or proofs).
	
	\begin{definition}
		Let $\mathbb{K} = \{\mc{A}_0,\mc{A}_1,\ldots\}$ be a finite or countable uniformly computable family of structures. $\mathbb{K}$ is $n$-friendly if the $n$-back-and-forth relations between tuples in structures from $\mathbb{K}$ are computable, i.e., if
		\[ \{ (i,\bar{a},j,\bar{b},m) : m \leq n \text{ and } (\mc{A}_i,\bar{a}) \leq_m (\mc{A}_j,\bar{b})\}\]
		is computable.
	\end{definition}
	\noindent In Chapter 15 of \cite{AshKnight}, Ash and Knight show that a number of structures, such as the standard copies of the computable ordinals, are $n$-friendly for all $n$.
	
	We will give another example which is the pair-of-structures theorem of Ash and Knight \cite{AshKnightPairs}. This can be thought of as an invariant version of the Louveau and Saint-Raymond separation theorem \cite{LSR87}. Given two structures $\mc{A}$ and $\mc{B}$ we ask how hard it is to tell $\mc{A}$ and $\mc{B}$ apart. If there is a $\Sigma_\alpha$ sentence true of $\mc{A}$ but not of $\mc{B}$, or vice versa, then this gives a $\bfSigma^0_\alpha$ way to tell them apart. The pair of structures theorem says that otherwise it is $\bfPi^0_\alpha$-hard to tell them apart. We begin by stating the relative or boldface version.
	\begin{theorem}[Ash and Knight]\label{thm:pairs-bold}
		Let $\mc{A}$ and $\mc{B}$ be structures such that $\mc{B} \leq_n \mc{A}$. Then for any $\bfPi^0_n$ set $S \subseteq 2^\mathbb{N}$, there is a continuous map $x \mapsto \mc{C}_x$ such that
		\[ \mc{C}_x \cong \begin{cases}
			\mc{A} & x \in S\\
			\mc{B} & x \notin S\\
		\end{cases}.\]
	\end{theorem}
	\noindent The continuous construction is computable relative to some parameter. To get a computable or lightface version of the theorem---which was the original version proved by Ash and Knight---we need the structures to be $n$-friendly.
	\begin{theorem}[Ash and Knight]\label{thm:pairs-light}
		Let $\mc{A}$ and $\mc{B}$ be computable structures such that $\mc{B} \leq_n \mc{A}$ and $\{\mc{A},\mc{B}\}$ is $n$-friendly. Then for any $\Pi^0_n$ set $S \subseteq 2^\mathbb{N}$, there is a computable map $x \mapsto \mc{C}_x$ such that
		\[ \mc{C}_x \cong \begin{cases}
			\mc{A} & x \in S\\
			\mc{B} & x \notin S\\
		\end{cases}.\]
	\end{theorem}
	Even if a structure or class of structures is not $n$-friendly, the back-and-forth relations
	\[ \{ (i,\bar{a},j,\bar{b},m) : m \leq n \text{ and } (\mc{A}_i,\bar{a}) \leq_m (\mc{A}_j,\bar{b})\}\]
	are always $\mathbf{0}^{(2n)}$-computable. This comes from simply writing out the inductive characterization of the back-and-forth relations. (Chen, Gonzalez, and Harrison-Trainor \cite{ChenGonzalezHT} have shown that the relations $\leq_n$ cannot be defined with fewer than $2n$ alternations of quantifiers, in the sense that $\{(\mc{A},\mc{B}) : \mc{A} \leq_n \mc{B}\}$ is $\bfPi^0_{2n}$-complete.) Thus, for example, in the computable pairs-of-structures theorem, even if $\{\mc{A},\mc{B}\}$ is not $n$-friendly then we can get a $\mathbf{0}^{(2n)}$-computable map $x \mapsto \mc{C}_x$. Similarly, if a computable structure is $\Delta^0_n$-categorical on a cone, then the base of the cone can be taken to be $\mathbf{0}^{(2n)}$. Though there are many known examples of structures which are not $n$-friendly, we generally do not have examples which are maximally unfriendly.
	
	\subsection{Friendly jump inversions}\label{sec:intro-jump-inversions}
	
	Friendly jump inversions are a standard technique in computable structure theory. See, for example, the text \cite{MonBook2} for the formal details. The idea can be illustrated most cleanly with a graph, though it works for any kind of structure.
	
	Let $\mc{G} = (V,E)$ be a graph. Fix two structures $\mc{A}$ and $\mc{B}$ in the same language $\mc{L}$, for example a common choice is $\mc{A} = (\omega,<)$ and $\mc{B} = (\omega^*,<)$ where $\omega^*$ is the reverse ordering. We can define a new structure $\mc{G}^*$ where we replace each edge of $\mc{G}$ by a copy of $\mc{A}$, and each non-edge of $\mc{G}$ by a copy of $\mc{B}$.
	\begin{center}
		\begin{tikzpicture}[
			scale=.82,
			vertex/.style={
				circle,
				fill=black,
				inner sep=1.7pt
			},
			gadget/.style={
				rectangle,
				draw,
				fill=white,
				minimum width=5.5mm,
				minimum height=4.5mm,
				inner sep=1pt,
				font=\scriptsize
			}
			]
			
			\coordinate (l1) at (0,1);
			\coordinate (l2) at (2,1);
			\coordinate (l3) at (1,0);
			\coordinate (l4) at (1,-1.4);
			
			\draw[thick]
			(l1) -- (l2)
			(l2) -- (l3)
			(l3) -- (l1)
			(l3) -- (l4);
			
			\node[vertex,label=above left:$u$]  at (l1) {};
			\node[vertex,label=above right:$v$] at (l2) {};
			\node[vertex,label=right:$w$]       at (l3) {};
			\node[vertex,label=below:$x$]       at (l4) {};
			
			\node[font=\Large] at (3.25,-.15) {$\Longrightarrow$};
			
			\coordinate (r1) at (4.5,1);
			\coordinate (r2) at (6.5,1);
			\coordinate (r3) at (5.5,0);
			\coordinate (r4) at (5.5,-1.4);
			
			\draw[thick]
			(r1) -- node[gadget] {$\mc{A}$} (r2);
			
			\draw[thick]
			(r1) -- node[gadget] {$\mc{A}$} (r3);
			
			\draw[thick]
			(r2) -- node[gadget] {$\mc{A}$} (r3);
			
			\draw[thick]
			(r3) -- node[gadget] {$\mc{A}$} (r4);
			
			\draw[thick]
			(r1) to[bend right=32]
			node[gadget,pos=.52] {$\mc{B}$}
			(r4);
			
			\draw[thick]
			(r2) to[bend left=32]
			node[gadget,pos=.52] {$\mc{B}$}
			(r4);
			
			\node[vertex,label=above left:$u$]  at (r1) {};
			\node[vertex,label=above right:$v$] at (r2) {};
			\node[vertex,label=right:$w$]       at (r3) {};
			\node[vertex,label=below:$x$]       at (r4) {};
			
		\end{tikzpicture}
	\end{center}
	To recover the original graph from $\mc{G}^*$, we need to be able to decide whether a given $\mc{L}$-structure is isomorphic to $\mc{A}$ or $\mc{B}$. For $\mc{A} = (\omega,<)$ and $\mc{B} = (\omega^*,<)$ this is not computable but requires a Turing jump. Thus the edge relation on $\mc{G}^*$ is definable by computable $\Delta_2$ formulas. Moreover, $\mc{G}^*$ has an $X$-computable copy if and only if $\mc{G}$ has an $X'$-computable copy. We can also use structures $\mc{A}$ and $\mc{B}$ which are even harder to tell apart in order to define the $n$th friendly jump inverse of $\mc{G}$, $\mc{G}^{(-n)}$, which has the property that the edge relation is computably $\Delta_{n+1}$-definable in $\mc{G}^{(-n)}$ and hence $\mc{G}^{(-n)}$ has an $X$-computable copy if and only if $\mc{G}$ has an $X^{(n)}$-computable copy. In all cases, the structures $\mc{A}$ and $\mc{B}$ that we use must be $(n+1)$-friendly.
	
	From the point of view of descriptive set theory, what is happening is as follows. Let $\mc{L}$ be the language of $\mc{G}$ and $\mc{L}^*$ the language of the friendly jump inversion. There is an injective continuous map
	\[ \phi : \Mod(\mc{L}) \to \Mod(\mc{L}^*) \]
	mapping $\mc{G}$ to $\mc{G}^{(-n)}$. The image of this map is Borel, and there is a Borel left-inverse
	\[ \psi : \text{Im}(\phi) \subseteq \Mod(\mc{L}^*) \to \Mod(\mc{L})\]
	mapping copies of $\mc{G}^{(-n)}$ to $\mc{G}$. The basic open sets of $\Mod(\mc{L})$ pull back to $\mathbf{\Delta}^0_{n+1}$ (in fact lightface $\Delta^0_{n+1}$) sets in $\Mod(\mc{L}^*)$. This is analogous to a ``reverse'' change of topology where we make the basic open sets in $\Mod(\mc{L})$ into $\Delta^0_{n+1}$ sets. In both the boldface and lightface settings the behaviour is the same.
	
	Friendly jump inversions are often used to ``push'' an example higher up the computability-theoretic hierarchy. For example, Goncharov, Harizanov, Knight, McCoy, R. Miller, and Solomon
	\cite{nonlow} showed that there is a structure that is $\Delta^0_{n+1}$-categorical but not relatively $\Delta^0_{n+1}$-categorical. They did this by taking a structure which is $\Delta^0_{1}$-categorical (also called computably categorical) but not relatively $\Delta^0_{1}$-categorical and taking the $n$th friendly jump inverse. This immediately yields a structure which is $\Delta^0_{n+1}$-categorical but not relatively $\Delta^0_{n+1}$-categorical.
	
	Similarly, Alvir, Csima, and Harrison-Trainor \cite{AlvirCsimaHT} showed that there is a computable structure with a $\Pi_2$ Scott sentence but no computable $\Sigma_4$ Scott sentence. They noted that using friendly jump inversions one can obtain a computable structure with a $\Pi_n$ Scott sentence but no computable $\Sigma_{n+2}$ Scott sentence. However, this is not optimal: One can only show that such a structure must have a computable $\Pi_{2n}$ Scott sentence. The complexity was increased by $n$ in both the non-effective setting (the Scott sentences of the structure) and the effective setting (the computable Scott sentences of the structure). To resolve this, we need a variation on friendly jump inversions that separates the effective and non-effective settings. We achieve this by taking a jump inversion with structures which are not friendly.
	
	\subsection{Unfriendly jump inversions}\label{sec:intro-unfriendly-jump-inversions}
	
	The unfriendly jump inversions introduced in this paper are like the friendly jump inversions just described except that one chooses the structures $\mc{A}$ and $\mc{B}$ differently. For the $n$th jump inversion, instead of choosing them to be $(n+1)$-friendly we choose them to be as far from $(n+1)$-friendly as possible. The main technical work goes into investigating what ``as far from $(n+1)$-friendly as possible'' means, and then producing such structures. The structures we use are given by the following theorem.
	
	\begin{restatable}{theorem}{invstructs}\label{thm:jump-inversion-structures}
		Fix $n \geq 1$. Let $S \subseteq \mathbb{N}$ be a complete $\Sigma^0_{2n+1}$ set. There are computable structures $\mc{A} \ncong \mc{B}$ and a uniformly computable sequence of structures $\mc{C}_i$ such that
		\[ i \in S \Longrightarrow \mc{C}_i \cong \mc{A} \]
		and
		\[ i \notin S \Longrightarrow \mc{C}_i \cong \mc{B}.\]
		Moreover:
		\begin{enumerate}
			\item $\mc{B} \nleq_{n+1} \mc{A}$, which by Theorem \ref{thm:separating-structures} implies that there is a $\mathbf{0}^{(n)}$-computable $\Sigma_{n+1}$ sentence $\varphi$ such that $\mc{A} \models \varphi$ and $\mc{B} \nmodels \varphi$;
			\item $\mc{A}$ and $\mc{B}$ have $\Pi_{n+2}$ Scott sentences.
		\end{enumerate} 
		This is uniform in $n$.
	\end{restatable}
	
	For the remainder of the paper, we use the notation $\mc{G}^{(-n)}$ for the $n$th \textit{unfriendly} jump inversion of $\mc{G}$ which is constructed in the same way as the friendly jump inversion but using the structures $\mc{A}$ and $\mc{B}$ given by the above theorem. These have the property that the edge relation is definable in $\mc{G}^{(-n)}$ by $\Delta_{n+1}$ formulas, but that these formulas are not computable; with computable formulas, the best one can do is to define the edge relation in a $\Delta_{2n+1}$ way.
	
	In terms of descriptive set theory, there is still an injective continuous map
	\[ \phi : \Mod(\mc{L}) \to \Mod(\mc{L}^*) \]
	mapping $\mc{G}$ to $\mc{G}^{(-n)}$ and there is a Borel left-inverse
	\[ \psi : \text{Im}(\phi) \subseteq \Mod(\mc{L}^*) \to \Mod(\mc{L})\]
	mapping copies of $\mc{G}^{(-n)}$ to $\mc{G}$. But now, the basic open sets of $\Mod(\mc{L})$ pull back to sets in $\Mod(\mc{L}^*)$ which are boldface $\mathbf{\Delta}^0_{n+1}$ but only lightface $\Delta^0_{2n+1}$. Thus again this is a ``reverse'' change of topology, but splitting the boldface and lightface behaviour.
	
	With friendly jump inversions, there is an $X^{(n)}$-computable copy of $\mc{G}$ if and only if there is an $X$-computable copy of $\mc{G}^{(-n)}$. Unfriendly jump inversions behave differently. Given a $\mathbf{0}^{(2n)}$-computable copy of $\mc{G}$, there is a computable copy of $\mc{G}^{(-n)}$. Given a $\mathbf{0}^{(n)}$-computable copy of $\mc{G}^{(-n)}$, there is a $\mathbf{0}^{(2n)}$-computable copy of $\mc{G}$. Combining these facts, if $\mc{G}^{(-n)}$ has a $\mathbf{0}^{(n)}$-computable copy, then it has a computable copy. In some sense, $\mathbf{0}$ and $\mathbf{0}^{(n)}$ are ``on the same footing''. This is a key insight that will be exploited especially in constructing a structure whose degree spectrum is the non-arithmetic degrees; an application like this is entirely different from the applications of friendly jump inversions. There are also applications which are similar to the applications of friendly jump inversions in the sense that we increase the complexity of an example, but now separate the increase in effective aspects of complexity from the increase in non-effective aspects of complexity. For example, whereas friendly jump inversions obtained a computable structure with a $\Pi_n$ Scott sentence but no computable $\Sigma_{n+2}$ Scott sentence, making use of unfriendly jump inversions we can obtain a computable structure with a $\Pi_n$ Scott sentence but no computable $\Sigma_{2n}$ Scott sentence. A full list of properties of the unfriendly jump inversions is given in Theorem \ref{thm:jump-inv}.
	
	\subsection{Applications}\label{sec:intro-applications}
	
	In this section we will describe some other applications of our method of unfriendly jump inversions, which will also provide some initial motivation for the method. Proofs of the results in this section appear at the end of the paper in Section \ref{sec:applications}, except for the fact that the non-arithmetic degrees are a degree spectrum, which will appear in Section \ref{sec:non-arithmetic-proof}.

	\subsubsection{Computability of back-and-forth relations}
	
	Unfriendly jump inversions use maximally unfriendly structures, i.e., those in which the back-and-forth relations are as far from computable as possible. Thus the following theorem is not at all surprising. 
	
	\begin{restatable}{theorem}{maxunfriendly}
		For each $n$ there is a computable structure $\mc{A}$ such that the $n$-back-and-forth relation $\{(\bar{a},\bar{b}) \in \mc{A}^{<\omega} \times \mc{A}^{< \omega} : \bar{a} \leq_n \bar{b}\}$ on $\mc{A}$ is $\Pi^0_{2n}$-complete and hence computes $\mathbf{0}^{(2n)}$.
	\end{restatable}
	
	\noindent This is typical of several applications of unfriendly jump inversions: For each of the various theorems in which $n$-friendliness is assumed our unfriendly jump inversions give a way to construct the best possible counterexamples in the absence of $n$-friendliness.
	
	\subsubsection{Computability of Scott sentences}
	
	Let $\mc{A}$ be a countable structure, and suppose that we want to characterize $\mc{A}$ up to isomorphism by writing down a description of $\mc{A}$ in some formal language. If we work in elementary first-order logic then, as a consequence of compactness, we cannot do this for most countable structures $\mc{A}$. However Scott \cite{Sco65} showed that if we use the infinitary logic $\mc{L}_{\omega_1 \omega}$ then we can characterize any countable structure up to isomorphism. For any countable structure $\mc{A}$ there is a sentence $\varphi$ of $\mc{L}_{\omega_1 \omega}$ such that for all countable $\mc{B}$,
	\[ \mc{B} \models \varphi \Longleftrightarrow \mc{A} \cong \mc{B}.\]
	We call such a sentence a \emph{Scott sentence} for $\mc{A}$. While being isomorphic to a fixed structure $\mc{A}$ is naively analytic, because $\mc{A}$ has a Scott sentence it is actually Borel (in $\Mod(\mc{L})$, the Polish space of presentations of countable $\mc{L}$-structures).
	
	One way to measure the complexity of a structure is the least complexity of a Scott sentence for that structure.
	
	\begin{definition}[Montalb\'an \cite{MonSR}, see also \cite{MonBook1}]
		The \textit{(unparametrized) Scott rank} of $\mc{A}$ is the least countable ordinal $\alpha$ such that $\mc{A}$ has a $\Pi_{\alpha+1}$ Scott sentence.
	\end{definition}
	
	\noindent This is a robust measure of the complexity of $\mc{A}$, and has several equivalent characterizations both externally in terms of other structures and internally in terms of automorphism orbits.
	
	\begin{theorem}[Montalb\'an \cite{MonSR}]
		Let $\mc{A}$ be a countable structure, and $\alpha$ a countable ordinal. The following are equivalent:
		\begin{enumerate}
			\item $\mc{A}$ has a $\Pi_{\alpha+1}$ Scott sentence.
			\item The set $\{\mc{B} \in \Mod(\mc{L}) : \mc{B} \cong \mc{A}\}$ of isomorphic copies of $\mc{A}$ is $\bfPi^0_{\alpha+1}$.
			\item Every automorphism orbit in $\mc{A}$ is $\Sigma_\alpha$-definable.
			\item $\mc{A}$ is (boldface) $\mathbf{\Delta}^0_\alpha$-categorical.
		\end{enumerate}
	\end{theorem}
	
	Note that if $\mc{A}$ has a $\Pi_n$ Scott sentence, then for all countable $\mc{B}$,
	\[ \mc{A} \leq_n \mc{B} \Longleftrightarrow \mc{B} \cong \mc{A}.\]
	Writing out the definition of the left-hand-side, we get a $\Pi_{2n}$ definition. Thus we get the well-known fact that if a computable structure has a $\Pi_n$ Scott sentence, then it has a computable $\Pi_{2n}$ Scott sentence.
	
	Alvir, Knight, and McCoy \cite{AlvirKnightMcCoy} showed that there is a computable structure with a $\Pi_2$ Scott sentence but no computable $\Pi_2$ Scott sentence, and this can be adjusted to show that for any $n$ there is a computable structure with a $\Pi_n$ Scott sentence but no computable $\Pi_n$ Scott sentence. Alvir, Csima, and Harrison-Trainor \cite{AlvirCsimaHT} showed that in the case $n = 2$, the above $\Pi_4$ bound is optimal.
	
	\begin{theorem}[Alvir, Csima, Harrison-Trainor \cite{AlvirCsimaHT}]\label{thm:no-comp-scott-2}
		There is a computable structure with a $\Pi_2$ Scott sentence but no computable $\Sigma_4$ Scott sentence.
	\end{theorem}
	
	\noindent They were only able to show that for each $n \geq 2$ there is a computable structure with a $\Pi_n$ Scott sentence but no computable $\Sigma_{n+2}$ Scott sentence. Using our unfriendly jump inversions, we resolve this question.
	
	\begin{theorem}
		For each $n \geq 2$ there is a computable structure with a $\Pi_n$ Scott sentence but no computable $\Sigma_{2n}$ Scott sentence.
	\end{theorem}
	
	\noindent Indeed our new techniques of unfriendly jump inversions allow us to obtain the general result almost directly from the case $n=2$ (at least, if we prove a slightly more general statement in the case $n=2$).
	
	\medskip
	
	Along with being an application of the method of unfriendly jump inversions, these results give some motivation to the method (and indeed the thought process below is how the author developed the method). Consider a computable structure $\mc{A}$ with a $\Pi_2$ Scott sentence but no computable $\Sigma_4$ Scott sentence. It has a computable $\Pi_4$ Scott sentence and (as we prove in Theorem \ref{thm:compute-scott-sentence}) a $\mathbf{0}''$-computable $\Pi_2$ Scott sentence. We can consider the complexity of certain index sets of structures isomorphic to $\mc{A}$. Let $(\mc{B}_i)_{i \in \mathbb{N}}$ be a computable list of the (possibly partial) computable $\mc{L}$-structures. Consider the index set
	\[ \text{$\Delta^0_1$-Copies}(\mc{A}) = \{ i \in \mathbb{N}: \mc{B}_i \cong \mc{A}\}\]
	of computable structures isomorphic to $\mc{A}$. Because $\mc{A}$ has a computable $\Pi_4$ Scott sentence, this set is $\Pi^0_4$. On the other hand, because $\mc{A}$ has no computable $\Sigma_4$ Scott sentence, one might imagine that this set is not $\Sigma^0_4$. Though this does not follow directly, by a slight modification of the same argument, one can build $\mc{A}$ so that this index set is $\Pi^0_4$-complete. Thus: deciding whether a computable structure is isomorphic to $\mc{A}$ is a $\Pi^0_4$-complete problem.
	
	On the other hand, consider a list $(\mc{C}_i^{\mathbf{0}''})_{i \in \mathbb{N}}$ of the (possibly partial) $\mathbf{0}''$-computable structures, and the index set
	\[ \Delta^0_3\text{-Copies}(\mc{A}) = \{ i \in \mathbb{N}: \mc{C}^{\mathbf{0}''}_i \cong \mc{A}\}\]
	of $\mathbf{0}''$-computable copies of $\mc{A}$. Because $\mc{A}$ has a $\mathbf{0}''$-computable $\Pi_2$ Scott sentence, this is also a $\Pi^0_4$ set! That is, the difficulty of deciding whether a computable structure is isomorphic to $\mc{A}$ is the same as the difficulty of deciding whether a $\mathbf{0}''$-computable structure is isomorphic to $\mc{A}$. This is highly unusual. Usually, if the set of computable structures isomorphic to $\mc{A}$ is $\Pi^0_4$ then we would expect that the set of $\mathbf{0}''$-computable structures isomorphic to $\mc{A}$ would be $\Pi^0_4$ relative to $\mathbf{0}''$, hence $\Pi^0_6$. This sort of phenomenon underlies the unfriendly jump inversions.
	
	\subsubsection{Definability of back-and-forth types}
	
	Given $\bar{a} \in \mc{A}$, the appropriate notion of types in countable model theory is often the $\Pi_n$-type of $\bar{a}$, which is the collection of $\Pi_n$ formulas true of $\bar{a}$. We have that $\bar{b} \in \mc{B}$ satisfies the $\Pi_n$-type of $\bar{a}$ if and only if $(\mc{A},\bar{a}) \leq_n (\mc{B},\bar{b})$.
	
	Within the countable structure $\mc{A}$ the $\Pi_n$-type of $\bar{a}$ is definable by a single $\Pi_n$ formula, that is, there is a $\Pi_n$ formula $\varphi_{\bar{a}}(\bar{x})$ such that, for $\bar{b} \in \mc{A}$, $\bar{a} \leq_n \bar{b}$ if and only if $\mc{A} \models \varphi_{\bar{a}}(\bar{b})$. This is essentially because there are only countably many other tuples to consider; see, e.g., Lemma II.62 of \cite{MonBook2}.
	
	On the other hand, there is no reason for a $\Pi_n$-type to be definable by a $\Pi_n$ formula in general among all tuples in all structures. In \cite{ChenGonzalezHT}, Chen, Gonzalez, and Harrison-Trainor showed that a $\Pi_n$-type is definable by a $\Pi_{n+2}$ formula, and that this is best possible. The simplest case (which is equivalent to the general case by naming constants) is to consider $0$-tuples, i.e., just structures. For each fixed structure $\mc{A}$ the set $\{ \mc{B} \in \Mod(\mc{L}) : \mc{A} \leq_n \mc{B}\}$ is (boldface) $\bfPi^0_{n+2}$, and there is a choice of $\mc{A}$ for which it is $\bfPi^0_{n+2}$-complete. This upper bound was not effective.
	
	If $\mc{A}$ is computable, what is the lightface complexity of this set? Equivalently, what is the least complexity of a computable sentence $\varphi$ such that $\mc{A} \leq_n \mc{B}$ if and only if $\mc{B} \models \varphi$? The set has the lightface upper bound $\Pi^0_{2n}$, witnessed by a computable $\Pi_{2n}$ sentence obtained by writing out the definition of the back-and-forth relations and replacing each quantifier over $\mc{A}$ with a countably infinite conjunction or disjunction over the elements of $\mc{A}$. In this paper, we show that this is best possible.
	
	\begin{restatable}{theorem}{backandforthtypes}
		For each $n \geq 2$, there is a computable structure $\mc{A}_n$ such that the set
		\[ \{ \mc{B} \in \Mod(\mc{L}) : \mc{A}_n \leq_n \mc{B}\}\]
		is not (lightface) $\Sigma^0_{2n}$.
	\end{restatable}
	
	\subsubsection{Separating structures}
	
	Suppose that $\mc{A}$ and $\mc{B}$ are computable structures and that $\mc{A} \nleq_n \mc{B}$. This means that there is a witnessing $\Pi_n$ sentence $\varphi$ such that $\mc{A} \models \varphi$ but $\mc{B} \nmodels \varphi$. What kind of non-computable separation can we find? Ash and Knight \cite{AshKnight} show that if $\mc{A}$ and $\mc{B}$ are $n$-friendly and have computable existential diagrams, then there is a computable  $\Pi_n$ sentence $\varphi$ such that $\mc{A} \models \varphi$ but $\mc{B} \nmodels \varphi$. In general, the best we can do is a computable $\Pi_{2n-1}$ sentence $\varphi$ such that $\mc{A} \models \varphi$ but $\mc{B} \nmodels \varphi$. Such a sentence $\varphi$ can be obtained by choosing $\bar{b} \in \mc{B}$ such that there is no $\bar{a} \in \mc{A}$ with $(\mc{A},\bar{a}) \geq_{n-1} (\mc{B},\bar{b})$, and then letting $\mc{C} \models \varphi$ if and only if for all $\bar{x}$ we have $(\mc{C},\bar{x}) \ngeq_{n-1} (\mc{B},\bar{b})$. We show that this is the best we can do:
	
	\begin{restatable}{theorem}{separating}
		Fix $n \geq 2$. There are computable structures $\mc{A}$ and $\mc{B}$ such that $\mc{B} \nleq_n \mc{A}$ but for any computable $\Sigma_{2n-1}$ sentence $\varphi$, if $\mc{B} \models \varphi$ then $\mc{A} \models \varphi$.
	\end{restatable}
	
	\subsubsection{Non-arithmetic is a degree spectrum}\label{sec:sw}
	
	It is well-known that if a set $A \subseteq \mathbb{N}$ is computable relative to every non-computable set then $A$ is computable \cite{KleenePost}. Nevertheless the non-computable sets do hold some non-trivial information in common. Answering a question of Lempp, Slaman \cite{Slaman} and Wehner \cite{Wehner} independently showed that there is a mathematical problem which can be solved by exactly the non-computable sets, namely the problem of computing a copy of some particular mathematical structure.
	
	\begin{theorem}[Slaman \cite{Slaman} and Wehner \cite{Wehner}]
		There is a countable structure $\mc{M}$ such that for any set $X$, $X$ computes an isomorphic copy of $\mc{M}$ if and only if $X$ is not computable.
	\end{theorem}
	
	\noindent By computing a copy of $\mc{M}$ we mean that we compute an isomorphic structure with domain $\mathbb{N}$ together with all of the functions, relations, and constants on $\mathbb{N}$.
	
	This result fits into the central program in computable structure theory of studying the \textit{degree spectra} of countable structures. Given a countable structure $\mc{A}$, the degree spectrum of $\mc{A}$ is the set of all Turing degrees that can compute an isomorphic copy of $\mc{A}$.\footnote{Often the degree spectrum is defined as the set of Turing degrees of atomic diagrams of copies of $\mc{A}$, but by a result of Knight \cite{Knight86} these definitions are equivalent except for the automorphically trivial structures.} This is a measurement of the difficulty of computing a copy of $\mc{A}$ and captures the information that is encoded in the structure $\mc{A}$. Classifying the degree spectra of countable structures is one of the central problems of computable structure theory. Though we have a large number of results showing that certain classes of Turing degrees are degree spectra, and some small number of results showing that particular classes are not degree spectra, we still seem far from a complete classification. Classifying the degree spectra means determining which computability-theoretic properties of sets of natural numbers can be encoded as the set of solutions to the problem of building a copy of a particular structure. A particular class of degrees being a degree spectrum is a $\bfSigma^1_1$---and in fact, by Scott's isomorphism theorem \cite{Sco65}, Borel---definition of that class which takes a very particular form, and so a classification would determine which upward-closed Borel sets of degrees admit a definition of this form. The above result of Slaman and Wehner shows that the non-computable degrees form a degree spectrum and hence are definable in this way.
	
	The simplest degree spectra are the upward cones, which consist of all degrees that compute a particular set $A$. Enumeration cones, the set of all degrees that can enumerate a set $A$, are also degree spectra. In addition to the Slaman-Wehner spectrum of non-computable degrees, other examples of degree spectra include the high degrees \cite{high}, the hyperimmune degrees \cite{hyperimmune}, the array non-recursive degrees and degrees without retraceable jump \cite{DGT}, the non-low$_\alpha$ degrees \cite{nonlow,KF}, and the non-$K$-trivial degrees \cite{KF}. On the other hand one can probably count on a single hand the results which show that some class of degrees is not a degree spectrum. The most well-known is that a non-trivial countable union of upward cones (or enumeration cones) is not a degree spectrum. (See Section V.3.2 of \cite{MonBook1} where Montalb\'an attributes this to Knight and her collaborators.) Furthermore, building on ideas of Andrews and Miller \cite{AndrewsMiller15}, Montalb\'an shows that the upward closure of any $F_\sigma$ set of reals is never a degree spectrum unless it is trivial, i.e., an enumeration cone. As a corollary of this, we get the result of Andrews and Miller that the DNC degrees, ML-random
	degrees, and PA degrees are not degree spectra.
	
	While there are uncountably many possible degree spectra, there are only countably many co-null degree spectra and so a good amount of work has concentrated on these. (To see this Greenberg, Montalb\'an, and Slaman \cite{GMS11} show---and, in that paper, note that Kalimullin and Nies had independently announced---that any such degree spectrum must include the Turing degree of Kleene's $\mc{O}$, the complete $\Pi^1_1$ set.) The Slaman-Wehner degree spectrum of the non-computable degrees is an example of such a degree spectrum.
	
	Greenberg, Montalb\'an, and Slaman \cite{nonhyp} asked whether the Slaman-Wehner theorem holds if we replace Turing reducibility by further higher reducibilities. They showed that it does hold for hyperarithmetic reducibility, where $A$ is hyperarithmetic in $B$ if $A$ is computable from some ordinal iterate $B^{(\beta)}$ of the Turing jump of $B$ for some $B$-computable $\beta$, or equivalently, if $A$ is $\Delta^1_1$ in $B$.
	
	\begin{theorem}[Greenberg, Montalb\'an, Slaman \cite{nonhyp}]
		There is a countable structure $\mc{M}$ such that for any set $X$, $X$ computes an isomorphic copy of $\mc{M}$ if and only if $X$ is not hyperarithmetic.
	\end{theorem}
	
	\noindent Thus the non-hyperarithmetic degrees form a degree spectrum. On the other hand, in \cite{GMS11} they show that this is not true for relative constructibility.
	
	Arithmetic reducibility lies between Turing (computable) reducibility and hyperarithmetic reducibility. A set $A$ is arithmetic in $B$ if there is a formula $\varphi$ in the language of arithmetic defining $A$ in $\mathbb{N}$ with parameter $B$, i.e., $A = \{ n : \mathbb{N} \models \varphi(n,B)\}$.
	Equivalently, $A$ is arithmetic in $B$ if $A$ is computable from some finite iterate $B^{(n)}$ of the Turing jump of $B$. Greenberg, Montalb\'an, and Slaman asked whether the Slaman-Wehner theorem holds for arithmetic reducibility: Are the non-arithmetic degrees a degree spectrum? For the past fifteen or so years, the non-arithmetic degrees have been the most prominent class of degrees for which we have not been able to answer this question. The main result of this paper is a resolution of this question.
	
	\begin{restatable}{theorem}{nonarithmetic}
		There is a countable structure $\mc{M}$ such that for any set $X$, $X$ computes an isomorphic copy of $\mc{M}$ if and only if $X$ is not arithmetic.
	\end{restatable}
	
	\noindent Thus the non-arithmetic degrees are a degree spectrum.
	
	One can think of the Slaman-Wehner result as saying that the non-$\Delta^0_1$ degrees are a degree spectrum. If we could show that for every $n$, the non-$\Delta^0_n$ degrees are a degree spectrum, and if there was a uniform list of examples, then non-arithmetic would be the degree spectrum of the multisorted structure combining these examples. Though this is true for the non-$\Delta^0_2$ degrees, which Kalimullin \cite{Kalimullin08} showed do form a degree spectrum by a modification of the Slaman-Wehner strategy, for $n \geq 3$ the non-$\Delta^0_n$ degrees are not a degree spectrum. This was shown by Andrews, Cai, Kalimullin, Lempp, Miller, and Montalb\'an \cite{lowercones}. For $n = 3$, for example, their strategy was as follows. If $\mc{M}$ was a structure whose degree spectrum contained all of the non-$\Delta^0_3$ degrees, then the existential theory of $\mc{M}$ would be c.e. They find a $\Delta^0_3$ degree $\mathbf{a}$ and a non-$\Delta^0_3$ degree $\mathbf{b}$ such that $\mathbf{a}' = \mathbf{b}'''$ and $\mathbf{b}''$ has $\mathbf{a}$-c.e.\ degree. Then $\mathbf{b}$ computes a copy of $\mc{M}$. They show, by an argument using Robinson low guessing, that $\mathbf{a}$ can guess at the third jump of $\mathbf{b}$ and thus build a copy of $\mc{M}$, showing that $\mc{M}$ has a $\Delta^0_3$ copy. However no strategy like this could be applied to show that the non-arithmetic degrees are not a degree spectrum: No finite number of jumps of an arithmetic degree can compute a non-arithmetic degree, let alone any of its jumps.
	
	On the other hand, the strategy of Greenberg, Montalb\'an, and Slaman \cite{nonhyp} for showing that the non-hyperarithmetic degrees are a degree spectrum relies on properties of infinite jumps. Namely, together with an argument involving non-standard jump hierarchies, they use the fact that a degree is hyperarithmetic if and only if it is low$_\alpha$ for some computable ordinal $\alpha$, as $\mathbf{0}^{(\alpha)}$ is low$_{\alpha \cdot \omega}$ because
	\[ \left(\mathbf{0}^{(\alpha)}\right)^{(\alpha \cdot \omega)} = \mathbf{0}^{(\alpha + \alpha \cdot \omega)} = \mathbf{0}^{(\alpha \cdot \omega)}.\] 
	The analogous statement is far from being the case within the arithmetic degrees: $\mathbf{0}'$ is arithmetic but not low$_n$ for any $n$.
	
	The key insight for solving the problem is that we must build a structure that, for each $n$, diagonalizes against $\mathbf{0}^{(n)}$-computable opponents. A standard friendly $n$-jump inversion allows us to work at the level $\mathbf{0}^{(n)}$, but our opponent now works at the level $\mathbf{0}^{(2n)}$. An unfriendly $n$-jump inversion allows us to work at the level $\mathbf{0}^{(2n)}$, because we started at the computable level, but our opponent will also only be able to work at level $\mathbf{0}^{(2n)}$. Thus the unfriendly jump inversions allow us to ``catch up''.

	\section{Computing optimal Scott sentences and separating sentences}\label{sec:computing-optimal-sentences}
	
	Suppose that $\mc{A}$ and $\mc{B}$ are computable structures and that $\mc{A} \nleq_n \mc{B}$, which means that there is a $\Pi_n$ sentence true of $\mc{A}$ but not of $\mc{B}$. Ash and Knight \cite{AshKnight} prove the theorem below which, as a corollary, gives conditions---that $\{\mc{A},\mc{B}\}$ be $n$-friendly with computable existential diagrams---under which we can compute such a sentence. The theorem is set in a more general context. For a tuple $\bar{a} \in \mc{A}$, the $\Pi_n$-type of $\bar{a}$ is known to be definable within $\mc{A}$ by a $\Pi_n$ formula, and moreover, this remains true as long as we consider only a countable set of structures. (As shown in \cite{ChenGonzalezHT}, $\Pi_n$-types are not definable by $\Pi_n$ formulas in general.) In the theorem, we get conditions under which the $\Pi_n$-types become definable by computable $\Pi_n$ formulas.
	
	\begin{theorem}[Theorem 15.2 of \cite{AshKnight}]\label{thm:ash-knight-sep}
		Let $(\mc{A}_i)_{i \in \mathbb{N}}$ be a computable list of $n$-friendly structures with uniformly computable existential diagrams. Then for any $m$ such that $1 \leq m \leq n$, and any tuple $\bar{a} \in \mc{A}_i$, we can find a computable $\Pi_m$ formula $\varphi^{m}_{i,\bar{a}}$ such that
		\[(\mc{A}_i,\bar{a}) \leq_m (\mc{A}_j,\bar{b}) \Longleftrightarrow \mc{A}_j \models \varphi^m_{i,\bar{a}}(\bar{b}).\]
	\end{theorem}
	
	Suppose that we do not know that the pair of structures is $n$-friendly or that they have computable existential diagrams. How hard is it to compute such formulas? Our main goal is the following theorem, along the way to which we will prove an analogue of Theorem \ref{thm:ash-knight-sep}.
	
	\begin{theorem}\label{thm:separating-structures}
		Let $\mc{A}$ and $\mc{B}$ be computable structures, and suppose that $\mc{A} \nleq_n \mc{B}$. Then there is a $\mathbf{0}^{(n-1)}$-computable $\Pi_n$ sentence $\varphi$ such that $\mc{A} \models \varphi$ and $\mc{B} \nmodels \varphi$.
	\end{theorem}
	
	\noindent We will also obtain in Theorem \ref{thm:compute-scott-sentence} that if a computable structure has a $\Pi_n$ Scott sentence then it has a  $\mathbf{0}^{(n)}$-computable $\Pi_n$ Scott sentence.
	
	In the proofs we will make multiple uses of the following fact which can be found as Proposition 7.12 of the book by Ash and Knight \cite{AshKnight}.
	
	\begin{proposition}\label{prop:absorb}
		Suppose that a formula $\varphi(\bar{x})$ is a disjunction of a $\Sigma^0_n$ set of (indices for) computable $\Sigma_n$ formulas. Then $\varphi(\bar{x})$ is equivalent to a computable $\Sigma_n$ formula, and we can find this computable $\Sigma_n$ formula uniformly. Similarly a conjunction of a $\Sigma^0_n$ set of computable $\Pi_n$ formulas is equivalent to a computable $\Pi_n$ formula.
	\end{proposition}
	
	\noindent Essentially the idea is that the complexity of which formulas are in the disjunction can be absorbed into the complexity of the formulas themselves. This proposition relativizes, and we will use it in relativized form.
	
	We now prove our theorem about defining $\Pi_n$ types, the analogue of Theorem \ref{thm:ash-knight-sep}. 
	
	\begin{theorem}\label{thm:internal-upper-bound}
		Let $(\mc{A}_i)_{i \in \mathbb{N}}$ be a computable list of computable structures. For each $n \geq 1$ and $\bar{a} \in \mc{A}_i$ there is a $\mathbf{0}^{(n)}$-computable $\Pi_n$ formula $\varphi^n_{i,\bar{a}}$ such that for all $\bar{b} \in \mc{A}_j$
		\[ (\mc{A}_i,\bar{a}) \leq_n (\mc{A}_j,\bar{b}) \; \Longleftrightarrow \mc{A}_j \models \varphi_{i,\bar{a}}^n(\bar{b}).\]
		Moreover, this is uniform in $i$, $\bar{a}$, and $n$.
	\end{theorem}
	\begin{proof}
		We argue inductively, noting that for each $n$ the statement of the theorem also implies that the set of pairs $\{ (i,\bar{a},j,\bar{b}) : (\mc{A}_i,\bar{a}) \leq_n (\mc{A}_j,\bar{b})\}$ is $\mathbf{0}^{(2n)}$-computable.
		
		For $n = 1$, we have $(\mc{A}_i,\bar{a}) \leq_1 (\mc{A}_j,\bar{b})$ if and only if every $\forall$-formula true of $\bar{a}$ is true of $\bar{b}$. For each $(\mc{A}_i,\bar{a})$, using $\mathbf{0}'$ we can list out the (finitary) $\forall$-formulas true of $\bar{a}$, and take the conjunction of these. This gives a $\mathbf{0}'$-computable $\Pi_1$ formula $\varphi_{i,\bar{a}}^1$ such that 
		\[ (\mc{A}_i,\bar{a}) \leq_1 (\mc{A}_j,\bar{b}) \; \Longleftrightarrow \mc{A}_j \models \varphi^1_{i,\bar{a}}(\bar{b}).\]
		
		Now we will verify the result for $n+1$ assuming we know it for $n$. For each $i$ and $\bar{c} \in \mc{A}_i$ let $\varphi^n_{i,\bar{c}}(\bar{x})$ be a $\mathbf{0}^{(n)}$-computable $\Pi_n$ formula such that
		\[ (\mc{A}_i,\bar{c}) \leq_n (\mc{A}_j,\bar{b}) \; \Longleftrightarrow \mc{A}_j \models \varphi_{i,\bar{c}}^n(\bar{b}).\]
		Then for each $i$ and each tuple $\bar{a} \in \mc{A}_i$ consider the $\Pi_{n+1}$ formula
		\[ \varphi^{n+1}_{i,\bar{a}}(\bar{x}) = \bigdoublewedge_j \bigdoublewedge_{\bar{b},\bar{b}' \in \mc{A}_j} \forall \bar{y} \bigdoublevee_{\bar{a}' \in \mc{A}_i} \begin{cases}
			\neg \varphi^{n}_{j,\bar{b}\bar{b}'}(\bar{x},\bar{y})	& (\mc{A}_j,\bar{b}\bar{b}') \nleq_n (\mc{A}_i,\bar{a}\bar{a}') \\
			\top & (\mc{A}_j,\bar{b}\bar{b}') \leq_n (\mc{A}_i,\bar{a}\bar{a}')
		\end{cases}.\]
		This notation means that the disjunct corresponding to $\bar{a}'$ is either $\neg \varphi^n_{j,\bar{b}\bar{b'}}(\bar{x},\bar{y})$ or $\top$, the true formula, depending on whether or not $(\mc{A}_j,\bar{b}\bar{b}') \leq_n (\mc{A}_i,\bar{a}\bar{a}')$.
		By the inductive hypothesis, for fixed $\bar{b},\bar{b}' \in \mc{A}_j$ and $\bar{a} \in \mc{A}_i$, for each $\bar{a}' \in \mc{A}_i$, deciding whether $(\mc{A}_j,\bar{b}\bar{b}') \leq_n (\mc{A}_i,\bar{a}\bar{a}')$ is $\mathbf{0}^{(2n)}$-computable. We use the proposition from \cite{AshKnight} highlighted above as Proposition \ref{prop:absorb} relativized to $\mathbf{0}^{(n+1)}$. The subformula
		\[ \bigdoublevee_{\bar{a}' \in \mc{A}_i} \begin{cases}
			\neg \varphi^n_{j,\bar{b}\bar{b}'}(\bar{x},\bar{y})	& (\mc{A}_j,\bar{b}\bar{b}') \nleq_n (\mc{A}_i,\bar{a}\bar{a}') \\
			\top & (\mc{A}_j,\bar{b}\bar{b}') \leq_n (\mc{A}_i,\bar{a}\bar{a}')
		\end{cases}\]
		is a disjunction over a $\mathbf{0}^{(2n)}$-computable (and hence $\Sigma^0_{2n+1}$, or equivalently $\Sigma^0_{n}$ relative to $\mathbf{0}^{(n+1)}$) set of indices for $\mathbf{0}^{(n)}$-computable $\Sigma_n$ formulas, and is thus equivalent to a  $\mathbf{0}^{(n+1)}$-computable $\Sigma_n$ formula. Thus $\varphi^{n+1}_{i,\bar{a}}(\bar{x})$ is equivalent to a $\mathbf{0}^{(n+1)}$-computable $\Pi_{n+1}$ formula. We must show that $\varphi^{n+1}_{i,\bar{a}}(\bar{x})$ is such that
		\[ (\mc{A}_i,\bar{a}) \leq_{n+1} (\mc{A}_j,\bar{b}) \; \Longleftrightarrow \mc{A}_j \models \varphi^{n+1}_{i,\bar{a}}(\bar{b}).\]
		\begin{itemize}
			\item First, we note that $\varphi^{n+1}_{i,\bar{a}}$ is true of $\bar{a}$, as for any $j$ and $\bar{b},\bar{b}' \in \mc{A}_j$ and $\bar{a}' \in \mc{A}_i$ (the value of $\bar{y}$), we have that the disjunct corresponding to $\bar{a}'$ is true: if $(\mc{A}_j,\bar{b}\bar{b}') \leq_n (\mc{A}_i,\bar{a}\bar{a}')$, then the disjunct is $\top$, and if  $(\mc{A}_j,\bar{b}\bar{b}') \nleq_n (\mc{A}_i,\bar{a}\bar{a}')$ then we have $\neg \varphi^n_{j,\bar{b}\bar{b}'}(\bar{a},\bar{a}')$. Since $\varphi^{n+1}_{i,\bar{a}}$ is $\Pi_{n+1}$ and true of $\bar{a}$, it is true of any $\bar{b} \in \mc{A}_j$ with $(\mc{A}_i,\bar{a}) \leq_{n+1} (\mc{A}_j,\bar{b})$.
			
			\item Suppose that $\varphi^{n+1}_{i,\bar{a}}$ is true of $\bar{b}$. Given $\bar{b}' \in \mc{A}_j$, taking the conjunct corresponding to $j$ and $\bar{b},\bar{b}' \in \mc{A}_j$ and taking $\bar{y} = \bar{b}'$, we have that
			\[\bigdoublevee_{\bar{a}'} \begin{cases}
				\neg \varphi^n_{j,\bar{b}\bar{b}'}(\bar{b},\bar{b}')	& (\mc{A}_j,\bar{b}\bar{b}') \nleq_n (\mc{A}_i,\bar{a}\bar{a}') \\
				\top & (\mc{A}_j,\bar{b}\bar{b}') \leq_n (\mc{A}_i,\bar{a}\bar{a}')
			\end{cases}\]
			Since $(\mc{A}_j,\bar{b}\bar{b}') \leq_n (\mc{A}_j,\bar{b}\bar{b}')$ we have $\varphi^n_{j,\bar{b}\bar{b}'}(\bar{b},\bar{b'})$ and so it must be that there is some disjunct that is $\top$, i.e., some $\bar{a}' \in \mc{A}_i$ with $(\mc{A}_j,\bar{b}\bar{b}') \leq_n (\mc{A}_i,\bar{a}\bar{a}')$. Thus $(\mc{A}_i,\bar{a}) \leq_{n+1} (\mc{A}_j,\bar{b})$.
		\end{itemize}
		This completes the inductive argument.
	\end{proof}
	
	We are now ready to prove our intended application.
	
	\begin{proof}[Proof of Theorem \ref{thm:separating-structures}]
		For $n = 1$ this is clear. Suppose $n \geq 2$. Suppose that $\mc{A}$ and $\mc{B}$ are computable structures and that $\mc{A} \nleq_n \mc{B}$. Then there is $\bar{b} \in \mc{B}$ such that for all $\bar{a} \in \mc{A}$ we have $(\mc{B},\bar{b}) \nleq_{n-1} (\mc{A},\bar{a})$. By the previous theorem (using just $\mc{A}$ and $\mc{B}$ as the sequence of structures), there is a $\mathbf{0}^{(n-1)}$-computable $\Pi_{n-1}$ formula $\psi(\bar{x})$ such that
		\[ (\mc{B},\bar{b}) \leq_{n-1} (\mc{A},\bar{a}) \; \Longleftrightarrow \mc{A} \models \psi(\bar{a}).\]
		Let ${\varphi} = \exists \bar{x}\; \psi(\bar{x})$. This is a $\mathbf{0}^{(n-1)}$-computable $\Sigma_{n}$ sentence such that $\mc{B} \models \varphi$ and $\mc{A} \nmodels \varphi$. The negation of $\varphi$ is a $\mathbf{0}^{(n-1)}$-computable $\Pi_{n}$ sentence true in $\mc{A}$ but not $\mc{B}$, proving the theorem.
	\end{proof}
	
	Since we have proved the theorem, we will also prove the following fact which is of independent interest (and which was also the author's initial motivation for the analysis above).
	
	\begin{theorem}\label{thm:compute-scott-sentence}
		If a computable structure $\mc{A}$ has a $\Pi_n$ Scott sentence then it has a $\mathbf{0}^{(n)}$-computable $\Pi_n$ Scott sentence.
	\end{theorem}
	\begin{proof}
		The case $n = 1$ is easy, so we can assume $n \geq 2$. It will be more convenient to prove the theorem for $n+1$ rather than $n$. Suppose that $\mc{A}$ has a $\Pi_{n+1}$ Scott sentence; we must show that it has a $\mathbf{0}^{(n+1)}$-computable Scott sentence. Since $\mc{A}$ has a $\Pi_{n+1}$ Scott sentence we know that for each tuple $\bar{a}$ there is a $\Sigma_n$ formula $\theta_{\bar{a}}(\bar{x})$ such that
		\[ \bar{a} \cong \bar{b} \; \Longleftrightarrow\; \bar{a} \leq_n \bar{b}\; \Longleftrightarrow\; \mc{A} \models \theta_{\bar{a}}(\bar{b}).\]
		We may assume that $\theta_{\bar{a}}(\bar{x})$ is of the form
		\[ \theta_{\bar{a}}(\bar{x}) = \exists \bar{y}\; \varphi^{n-1}_{\bar{a}\bar{a}'}(\bar{x},\bar{y}) \]
		for some $\bar{a}'$, where $\varphi^{n-1}_{\bar{a}\bar{a}'}$ is the $\mathbf{0}^{(n-1)}$-computable $\Pi_{n-1}$ formula obtained from Theorem \ref{thm:internal-upper-bound} with $\bar{a}\bar{a}' \leq_{n-1} \bar{b} \bar{b}'$ if and only if $\mc{A} \models \varphi^{n-1}_{\bar{a}\bar{a}'}(\bar{b},\bar{b}')$. The reason for this is as follows. If $\theta_{\bar{a}}(\bar{x})$ was of the form $\bigdoublevee_k \;\exists \bar{y}_k \; \psi_k(\bar{x},\bar{y}_k)$ where $\psi_k$ is $\Pi_{n-1}$ then take any $k$ and $\bar{a}'$ such that $\mc{A} \models \psi_k(\bar{a},\bar{a}')$. Since $\psi_k(\bar{x},\bar{y})$ is $\Pi_{n-1}$ and holds of $\bar{a},\bar{a}'$, $\varphi^{n-1}_{\bar{a}\bar{a}'}(\bar{x},\bar{y})$ implies $\psi_k(\bar{x},\bar{y})$ within $\mc{A}$. Thus $\exists \bar{y}\; \varphi^{n-1}_{\bar{a}\bar{a}'}(\bar{x},\bar{y})$ implies $\theta_{\bar{a}}(\bar{x})$ within $\mc{A}$. On the other hand, if $\mc{A} \models \theta_{\bar{a}}(\bar{b})$ then there is an automorphism taking $\bar{a}$ to $\bar{b}$, and so $\mc{A} \models \exists \bar{y}\;\varphi^{n-1}_{\bar{a}\bar{a}'}(\bar{b},\bar{y})$. Thus $\exists \bar{y}\;\varphi^{n-1}_{\bar{a}\bar{a}'}(\bar{b},\bar{y})$ and $\theta_{\bar{a}}(\bar{x})$ are equivalent within $\mc{A}$.
		
		This formula $\theta_{\bar{a}}(\bar{x})$ is a $\mathbf{0}^{(n-1)}$-computable $\Sigma_n$ formula as $\varphi^{n-1}_{\bar{a}\bar{a}'}(\bar{x},\bar{y})$ is a $\mathbf{0}^{(n-1)}$-computable $\Pi_{n-1}$ formula. Given $\bar{a}$, we may find the appropriate $\bar{a}'$ and thus the formula $\theta_{\bar{a}}(\bar{x})$ using $\mathbf{0}^{(2n)}$ by searching for a $\bar{a}'$ such that for all $\bar{b},\bar{b}'$, if $\mc{A} \models \varphi^{n-1}_{\bar{a}\bar{a}'}(\bar{b},\bar{b}')$ then $\bar{a} \leq_n \bar{b}$. For each $\bar{a}'$, this can be done using $\mathbf{0}^{(2n)}$ since we can decide whether $\mc{A} \models \varphi^{n-1}_{\bar{a}\bar{a}'}(\bar{b},\bar{b}')$ using $\mathbf{0}^{(2n-2)}$ and deciding whether $\bar{a} \leq_n \bar{b}$ is $\Pi^0_{2n}$.
		
		Now we can build the canonical Scott sentence for $\mc{A}$. For tuples $\bar{a}$ and $\bar{x}$ of the same length, write $D(\bar{x})=D(\bar{a})$ for the conjunction of the atomic and negated atomic formulas whose truth values on $\bar{x}$ agree with their truth values on $\bar{a}$. If the language is infinite, we use only the first $|\bar{x}|$-many symbols in the fixed enumeration of the language. The formula $\theta_{\varnothing}$ is the member of the family $\theta_{\bar{a}}$ corresponding to the empty tuple $\bar{a}=\varnothing$; in particular, it is a sentence true in $\mc{A}$. Let
		\[
		\chi = \theta_{\varnothing} \wedge \bigdoublewedge_{\bar{a} \in \mc{A}} \forall \bar{x} \left[\theta_{\bar{a}}(\bar{x}) \longrightarrow \left( D(\bar{x}) = D(\bar{a}) \wedge \bigdoublewedge_{a' \in \mc{A}} \exists y \; \theta_{\bar{a}a'}(\bar{x},y) \wedge \forall y \bigdoublevee_{a' \in \mc{A}} \theta_{\bar{a}a'}(\bar{x},y) \right)\right].
		\]
		The conjunct $\theta_{\varnothing}$ ensures that every model of $\chi$ enters the intended back-and-forth system at its root. The sentence $\chi$ is a $\mathbf{0}^{(2n)}$-computable $\Pi_{n+1}$ formula. We now use the proposition of Ash and Knight \cite{AshKnight} highlighted above as Proposition \ref{prop:absorb} to improve this. For fixed $\bar{a}$, the disjunction 
		\[\bigdoublevee_{a' \in \mc{A}} \theta_{\bar{a}a'}(\bar{x},y)\]
		is a disjunction of a $\mathbf{0}^{(2n)}$-computable, and hence $\Sigma^0_{n}$ relative to $\mathbf{0}^{(n+1)}$, set of (indices for) $\mathbf{0}^{(n-1)}$-computable $\Sigma_n$ formulas. Thus the disjunction is equivalent to a $\mathbf{0}^{(n+1)}$-computable $\Sigma_n$ formula. Also for fixed $\bar{a}$, the conjunction 
		\[\bigdoublewedge_{a' \in \mc{A}} \exists y\; \theta_{\bar{a}a'}(\bar{x},y) \]
		is a conjunction of a $\mathbf{0}^{(2n)}$-computable, and hence $\Pi^0_{n+1}$ relative to $\mathbf{0}^{(n)}$, set of (indices for) $\mathbf{0}^{(n-1)}$-computable $\Sigma_n$ (hence $\mathbf{0}^{(n-1)}$-computable $\Pi_{n+1}$) formulas. Thus the conjunction is equivalent to a $\mathbf{0}^{(n)}$-computable $\Pi_{n+1}$ formula. It is $\mathbf{0}^{(2n)}$-computable to find these formulas given $\bar{a}$.
		
		Finally, $\chi$ is a conjunction of a $\mathbf{0}^{(2n)}$-computable, hence $\Pi^0_{n+1}$ relative to $\mathbf{0}^{(n+1)}$, set of (indices for) $\mathbf{0}^{(n+1)}$-computable $\Pi_{n+1}$ formulas. Thus $\chi$ is equivalent to a $\mathbf{0}^{(n+1)}$-computable $\Pi_{n+1}$ formula.
	\end{proof}
	
	\section{Unfriendly jump inversions}\label{sec:unfriendly-jump-inversions}
	
	We begin by describing the related methods of friendly jump inversion and friendly Marker extensions as they are generally used. We follow, for example, Chapter X.3 of \cite{MonBook1}. Given $n$, let $\mc{A}$ and $\mc{B}$ be computable structures such that $\mc{A} \equiv_n \mc{B}$ but $\mc{A} \nleq_{n+1} \mc{B}$ and $\mc{B} \nleq_{n+1} \mc{A}$. Suppose further that $\mc{A}$ and $\mc{B}$ have computable $\Pi_{n+2}$ Scott sentences, are relatively $\Delta^0_{n+1}$-categorical, and that $\mc{A}$ and $\mc{B}$ are $(n+1)$-friendly, and that there are computable $\Pi_{n+1}$ sentences satisfied by $\mc{A}$ but not by $\mc{B}$ and by $\mc{B}$ but not by $\mc{A}$. Montalb\'an asks in addition that $\mc{A}$ and $\mc{B}$ be rigid.
	
	If we have a structure $\mc{M}$ which we know is isomorphic to either $\mc{A}$ or $\mc{B}$, we can figure out which in a $\Delta^0_{n+1}$ way relative to $\mc{M}$ by asking whether $\mc{M}$ satisfies the computable $\Pi_{n+1}$ sentence true of $\mc{A}$ but not $\mc{B}$, or the computable $\Pi_{n+1}$ sentence true of $\mc{B}$ but not $\mc{A}$. On the other hand, distinguishing between $\mc{A}$ and $\mc{B}$ is $\Delta^0_{n+1}$-hard: By the Ash and Knight pairs-of-structures theorem, in the form of Theorem 18.7 of \cite{AshKnight}, for any $\Delta^0_{n+1}$ set $S$ there is a uniformly computable sequence of structures $\mc{C}_i$ such that
	\[ i \in S \Longleftrightarrow \mc{C}_i \cong \mc{A}\]
	and
	\[ i \notin S \Longleftrightarrow \mc{C}_i \cong \mc{B}.\]
	This also relativizes.
	
	Given a structure, e.g., for the sake of simplicity a graph $\mc{G}$, we can define the $n$th friendly jump inversion $\mc{G}^{(-n)}$ of $\mc{G}$ by replacing each edge in $\mc{G}$ by a copy of $\mc{A}$, and each non-edge in $\mc{G}$ by a copy of $\mc{B}$. For a more general $k$-ary relation, for each $k$-tuple we attach a copy of either $\mc{A}$ or $\mc{B}$ depending on whether the relation holds of the $k$-tuple or not. The idea is that the edge relation of $\mc{G}$ becomes a definable relation in $\mc{G}^{(-n)}$, but the complexity has increased to being $\Delta^0_{n+1}$. There are various different forms of friendly jump inversions, with different properties, but typically one proves something like the following facts.
	\begin{theorem}\label{thm:traditional-jump-inversion}
		If $\mc{A}$ and $\mc{B}$ are rigid, then the $n$th friendly jump inversion $\mc{G}^{(-n)}$ is effectively bi-interpretable with $\mc{G}$.\footnote{We will not define these terms, but the reader can find them in \cite{MonBook1} and \cite{MonBook2}. In any case, properties (1)-(5) give a sufficient idea of what is going on.} Even if $\mc{A}$ and $\mc{B}$ are not rigid, we have the following:
		\begin{enumerate}
			\item For every $\mathbf{d}$-computable $\Pi_{n+k}$ formula $\varphi$ there is a $\mathbf{d}$-computable $\Pi_k$ formula $\varphi_*$ such that $\mc{G} \models \varphi_*$ if and only if $\mc{G}^{(-n)} \models \varphi$.
			\item For every $\mathbf{d}$-computable $\Pi_{k}$ formula $\varphi$ there is a $\mathbf{d}$-computable $\Pi_{n+k}$ formula $\varphi^*$ such that $\mc{G} \models \varphi$ if and only if $\mc{G}^{(-n)} \models \varphi^*$.
			\item If $\mc{G}$ has a $\Pi_k$ Scott sentence then $\mc{G}^{(-n)}$ has a $\Pi_{n+k}$ Scott sentence.
			\item Given a $\mathbf{d}$-computable copy of $\mc{G}^{(-n)}$ there is a $\mathbf{d}^{(n)}$-computable copy of $\mc{G}$.
			\item Given a $\mathbf{d}^{(n)}$-computable copy of $\mc{G}$ there is a $\mathbf{d}$-computable copy of $\mc{G}^{(-n)}$.
		\end{enumerate}
	\end{theorem}
	
	For a friendly Marker extension we have a similar process, except that for different relations in the language we might use different structures $\mc{A}$ and $\mc{B}$, or we might leave certain relations alone. So, for example, in a language including relations $R_1,R_2,\ldots$ by taking a friendly Marker extension we might make the relation $R_n$ into a $\Delta_{n+1}$-definable relation.
	
	Instead of using friendly structures, for our unfriendly jump inversions we will use structures $\mc{A}$ and $\mc{B}$ that are maximally unfriendly. These structures are given by the following theorem which we will prove later.
	
	\invstructs*
	
	Given the theorem for some particular value of $n$, we can use these structures $\mc{A}$ and $\mc{B}$ to define a particular jump inversion operation which we will use for the construction of a structure with non-arithmetic degree spectrum. This jump inversion operation will be described in full detail in Section \ref{sec:jump-inversion}, where we also prove the desired properties (as the analogue of Theorem \ref{thm:traditional-jump-inversion}). Montalb\'an notes that even if two computable structures are not $n$-friendly, they become $n$-friendly relative to $\mathbf{0}^{(2n)}$. Thus we get different properties depending on what degree we are working relative to. The key difference from the standard friendly jump inversions is:
	\begin{enumerate}
		\item Given a $\mathbf{0}^{(2n)}$-computable copy of $\mc{G}$, there is a computable copy of $\Inv{n}{\mc{G}}$.
		\item Given a $\mathbf{d}$-computable copy of $\Inv{n}{\mc{G}}$ for $\mathbf{d} \geq_T \mathbf{0}^{(n)}$, there is a $\mathbf{d}^{(n)}$-computable copy of $\mc{G}$.
	\end{enumerate}
	These might seem contradictory, but (1) does not relativize and (2) is only true for $\mathbf{d} \geq_T \mathbf{0}^{(n)}$. These imply, for example, that if $\mc{G}^{(-n)}$ has a $\mathbf{0}^{(n)}$-computable copy then it has a computable copy. This is usually not true for friendly jump inversions. 
	
	The proof of Theorem \ref{thm:jump-inversion-structures} will be an inductive argument which we give in Section \ref{sec:inductive}. We begin by proving the case $n = 1$ as Lemma \ref{lem:n1}. To prove the case $n+1$, we first use Lemma \ref{lem:n1} relativized to $\mathbf{0}^{(2n)}$, and then use the jump inversion operation obtained from the inductive hypothesis, which is the case $n$ of Theorem \ref{thm:jump-inversion-structures}, to obtain Theorem \ref{thm:jump-inversion-structures} for $n+1$.
	
	\subsection{Construction and properties of unfriendly jump inversions}\label{sec:jump-inversion}
	
	Suppose that we are given Theorem \ref{thm:jump-inversion-structures} for some value of $n$. Let $S$ be a complete $\Sigma^0_{2n+1}$ set, and fix computable structures $\mc{A} \ncong \mc{B}$ with $\Pi_{n+2}$ Scott sentences, a computable sequence of structures $\mc{C}_i$, and a $\mathbf{0}^{(n)}$-computable $\Sigma_{n+1}$ sentence $\varphi$ as in Theorem \ref{thm:jump-inversion-structures}.
	
	Consider a graph $\mc{G}$ (or, by simple modifications, a structure in any relational language). Define $\Inv{n}{\mc{G}}$ to be the structure consisting of the vertices of $\mc{G}$, and associated to each pair $(u,v)$ of vertices we have an ordered pair of structures isomorphic to $(\mc{A},\mc{B})$ if there is an edge $u - v$, and isomorphic to $(\mc{B},\mc{A})$ if there is no edge $u - v$. More formally, let $\mc{L}$ be the language of $\mc{A}$ and $\mc{B}$. We work in the language $\mc{L}^* = \mc{L} \cup \{U,P,Q\}$ where $U$ is a new unary relation and $P$ and $Q$ are new ternary relations. $\Inv{n}{\mc{G}}$ will be the structure defined as follows. The domain of $\Inv{n}{\mc{G}}$ will consist of the elements of $\mc{G}$ together with, for each unordered pair $u,v \in \mc{G}$, disjoint infinite sets $C_{u,v}$ and $D_{u,v}$. $U$ will hold only of the elements of $\mc{G}$. We put $P(u,v,x)$ if and only if $x \in C_{u,v}$ and $Q(u,v,x)$ if and only if $x \in D_{u,v}$. The relations of $\mc{L}$ will hold only of tuples of elements all in the same set $C_{u,v}$ or $D_{u,v}$, and we will have $C_{u,v} \cong \mc{A}$ and $D_{u,v} \cong \mc{B}$ if there is an edge $u - v$, and $C_{u,v} \cong \mc{B}$ and $D_{u,v} \cong \mc{A}$ if there is no edge $u-v$.
	
	Note that $\mc{G}$ is interpretable in $\Inv{n}{\mc{G}}$ using computable $\Delta_{2n+1}$ formulas (in the sense of $\mc{L}_{\omega_1 \omega}$ interpretations as introduced in \cite{HTMM} and \cite{HTM}; see Section VII.8 of \cite{MonBook2}). Namely, the domain of $\mc{G}$ in $\Inv{n}{\mc{G}}$ is defined by $U$, and we define the edge relation on pairs $u,v \in U$ using computable $\Delta_{2n+1}$ formulas, as
	\[ u E v \Longleftrightarrow P(u,v,\cdot) \models \varphi \Longleftrightarrow Q(u,v,\cdot) \models \neg \varphi.\]
	Here, by $P(u,v,\cdot) \models \varphi$ we mean that the $\mc{L}$-structure on the set $\{x : P(u,v,x)\}$ is a model of $\varphi$. (Since $\varphi$ is a $\mathbf{0}^{(n)}$-computable $\Sigma_{n+1}$ sentence it is equivalent to a computable $\Sigma_{2n+1}$ sentence by Proposition \ref{prop:absorb}.)  Relative to $\mathbf{0}^{(n)}$, these definitions become $\Delta_{n+1}$. Note that this interpretation of $\mc{G}$ in $\Inv{n}{\mc{G}}$ is uniform in $\mc{G}$. Hence $\mc{G} \cong \mc{H}$ if and only if $\Inv{n}{\mc{G}} \cong \Inv{n}{\mc{H}}$.
	
	We prove several properties of this jump inversion. Some of them, like (3), (4), and (5), are properties that any form of jump inversion should satisfy. (1) and (2) are special properties of unfriendly jump inversions.
	
	\begin{theorem}\label{thm:jump-inv} Unfriendly jump inversions enjoy the following properties:
		\begin{enumerate}
			\item Given a $\mathbf{0}^{(2n)}$-computable copy of $\mc{G}$, there is a computable copy of $\Inv{n}{\mc{G}}$.
			\item For every $\mathbf{d} \geq_T \mathbf{0}^{(n)}$, given a $\mathbf{d}$-computable copy of $\Inv{n}{\mc{G}}$, there is a $\mathbf{d}^{(n)}$-computable copy of $\mc{G}$.
			\item Given a $\Sigma_\ell$ sentence $\varphi$, there is a $\Sigma_{n+\ell}$ sentence $\varphi^*$ such that for all $\mc{G}$, $\mc{G} \models \varphi$ if and only if $\Inv{n}{\mc{G}} \models \varphi^*$.
			\item If $\mc{G}$ has a $\Pi_\ell$ Scott sentence for $\ell \geq 2$, then $\Inv{n}{\mc{G}}$ has a $\Pi_{n+\ell}$ Scott sentence.
			\item If $\mc{G}$ has a $\Sigma_\ell$ Scott sentence for $\ell \geq 3$, then $\Inv{n}{\mc{G}}$ has a $\Sigma_{n+\ell}$ Scott sentence.
		\end{enumerate}
		Moreover, (1) and (2) are uniform.
	\end{theorem}
	\begin{proof}
		
		For (1): Given a $\mathbf{0}^{(2n)}$-computable copy of $\mc{G}$, the edge relation of $\mc{G}$ is both $\Sigma^0_{2n+1}$ and $\Pi^0_{2n+1}$. Thus there are computable functions $f$ and $g$ such that
		\[ u E v \Longleftrightarrow f(u,v) \in S \]
		and
		\[ u E v \Longleftrightarrow g(u,v) \notin S.\]
		Then we can build a computable copy of $\Inv{n}{\mc{G}}$: given $u$ and $v$, associate to them the pair of structures $(\mc{C}_{f(u,v)},\mc{C}_{g(u,v)})$. Then
		\[ u E v \Longleftrightarrow \mc{C}_{f(u,v)} \cong \mc{A} \text{ and } \mc{C}_{g(u,v)} \cong \mc{B}\]
		and
		\[ u \not E v \Longleftrightarrow \mc{C}_{f(u,v)} \cong \mc{B} \text{ and } \mc{C}_{g(u,v)} \cong \mc{A}.\]
		
		\medskip
		
		For (2): Fix $\mathbf{d} \geq_T \mathbf{0}^{(n)}$ and suppose that we are given a $\mathbf{d}$-computable copy of $\Inv{n}{\mc{G}}$. The $\mathbf{0}^{(n)}$-computable $\Sigma_{n+1}$ sentence $\varphi$ separating $\mc{A}$ from $\mc{B}$ is also $\mathbf{d}$-computable. Given vertices $u$ and $v$, it is $\Delta^0_{n+1}$ relative to $\mathbf{d}$ to decide whether they have associated to them a copy of $(\mc{A},\mc{B})$ or $(\mc{B},\mc{A})$. A $\Delta^0_{n+1}$ decision relative to $\mathbf{d}$ is computable from $\mathbf{d}^{(n)}$, and so we can use $\mathbf{d}^{(n)}$ to compute a copy of $\mc{G}$.
		
		\medskip
		
		For (3): Given a $\Sigma_\ell$ sentence $\theta$, we can create a new sentence $\theta^*$ by replacing every instance of $u E v$ with the  $\Sigma_{n+1}$ formula $P(u,v,\cdot)\models \varphi$ or the $\Pi_{n+1}$ formula $Q(u,v,\cdot) \models \neg \varphi$, and every instance of $u \not E v$ replaced by the $\Sigma_{n+1}$ formula $Q(u,v,\cdot)\models \varphi$ or the $\Pi_{n+1}$ formula $P(u,v,\cdot) \models \neg \varphi$. We also replace each quantifier in $\theta$ by a quantifier restricted to $U$. Thus $\theta^*$ is a $\Sigma_{n+\ell}$ sentence, and $\mc{G} \models \theta$ if and only if $\Inv{n}{\mc{G}} \models \theta^*$.
		
		\medskip
		
		For (4): We can write a $\Pi_{n+\ell}$ Scott sentence for $\Inv{n}{\mc{G}}$ as follows:
		\begin{enumerate}[label=(\alph*)]
			\item The following $\Pi_2$ sentences:
			\begin{enumerate}[label=(\roman*)]
				\item If $P(u,v,x)$ or $Q(u,v,x)$ then $u \neq v$, $u,v \in U$, and $x \notin U$.
				\item $P(u,v,x)$ if and only if $P(v,u,x)$ and $Q(u,v,x)$ if and only if $Q(v,u,x)$.
				\item For all $x \notin U$, there are $u$ and $v$ such that $P(u,v,x)$ or $Q(u,v,x)$.
				\item If $P(u,v,x)$ (or $Q(u,v,x)$) and $P(u',v',x)$ (or $Q(u',v',x)$) then $\{u,v\} = \{u',v'\}$.
				\item for all $u,v,x$ it is not the case that both $P(u,v,x)$ and $Q(u,v,x)$.
			\end{enumerate}
			\item For all distinct $u,v \in U$, either $P(u,v,\cdot) \cong \mc{A}$ and $Q(u,v,\cdot) \cong \mc{B}$, or alternatively $P(u,v,\cdot) \cong \mc{B}$ and $Q(u,v,\cdot) \cong \mc{A}$. Since $\mc{A}$ and $\mc{B}$ have $\Pi_{n+2}$ Scott sentences, this is $\Pi_{n+2}$. (The symbols from the language of $\mc{A}$ and $\mc{B}$ are restricted to these domains.)
			
			\item The $\Pi_{\ell}$ Scott sentence of $\mc{G}$, but with quantifiers restricted to $U$ and every instance of $u E v$ replaced by the $\Sigma_{n+1}$ formula $P(u,v,\cdot)\models \varphi$ or the $\Pi_{n+1}$ formula $Q(u,v,\cdot) \models \neg \varphi$, and every instance of $u \not E v$ replaced by $\Sigma_{n+1}$ formula $Q(u,v,\cdot)\models \varphi$ or the $\Pi_{n+1}$ formula $P(u,v,\cdot) \models \neg \varphi$. This is $\Pi_{\ell+n}$.
		\end{enumerate}
		Here (a) and (b) imply that the structure is $\Inv{n}{\mc{H}}$ for some $\mc{H}$, and (c) expresses that $\mc{H} \cong \mc{G}$.
		
		\medskip
		
		For (5): Similar to (4), but (c) is now $\Sigma_{\ell+n}$.
	\end{proof}
	
	\subsection{Constructing the required structures}\label{sec:inductive}
	
	We begin by proving as a lemma the case $n=1$ of Theorem \ref{thm:jump-inversion-structures}.
	
	\begin{lemma}\label{lem:n1}
		Let $S \subseteq \mathbb{N}$ be a complete $\Sigma^0_{3}$ set. There are computable structures $\mc{A} \ncong \mc{B}$ and a uniformly computable sequence of structures $\mc{C}_i$ such that
		\[ i \in S \Longrightarrow \mc{C}_i \cong \mc{A} \]
		and
		\[ i \notin S \Longrightarrow \mc{C}_i \cong \mc{B}.\]
		Moreover:
		\begin{enumerate}
			\item $\mc{B} \nleq_2 \mc{A}$ and so there is a $\mathbf{0}'$-computable $\Sigma_{2}$ sentence $\varphi$ such that $\mc{A} \models \varphi$ and $\mc{B} \nmodels \varphi$;
			\item $\mc{A}$ and $\mc{B}$ have $\Pi_{3}$ Scott sentences.
		\end{enumerate} 
	\end{lemma}
	\begin{proof}
		These structures $\mc{A}$ and $\mc{B}$ will be ``bouquet graphs''. A bouquet graph consists of a number of different connected components (``flowers'') with each connected component having a central vertex and then some number of loops attached to the central vertex. If there is a loop of length, say, $n+3$, then we think of the component as having been labeled with the number $n$. Thus each flower corresponds to some subset of $\mathbb{N}$ consisting of the labels it has received, and the graph corresponds to a collection of subsets of $\mathbb{N}$. We can think of the labels as being c.e.\ since as the graph is being constructed we can always add new loops/labels, but once a label is added it cannot be removed. It will be more convenient for us to think of the structures as consisting of infinitely many points which are labeled with c.e.\ labels coming from the set of labels $\ell_n$. Putting a label $\ell_n$ on an element means to add a loop of length $n+3$ to the flower whose center is that element. 
		
		To make $\mc{A}$ and $\mc{B}$ have $\Pi_3$ Scott sentences we must ensure that each automorphism orbit is definable by a $\Sigma_2$ formula, which we will accomplish by making sure that for each point $x$ there is a finite set $P_x$ of labels and a possibly infinite set $N_x$ of labels such that whenever any $y$ satisfies all of the labels in $P_x$ and none of the labels in $N_x$, then $y$ satisfies the same set of labels as $x$ (and hence they are in the same automorphism orbit). We also need $\mc{B} \nleq_2 \mc{A}$, and so we need a $\Sigma_2$ sentence $\varphi$ true of $\mc{A}$ but not of $\mc{B}$. To this end, we will make sure that $\mc{A}$ has an element $a$ such that every element $b \in \mc{B}$ has a label that $a$ does not have.
		
		We reserve a special label $\lkill$, which we call a killing label. Whenever any element has the label $\lkill$, it will also have all other labels. During the construction we may say that we kill an element; by this, we mean that we put the label $\lkill$ on that element, and also every other label.
		
		We build, stage-by-stage, computable structures $\mc{B}$, $\mc{A}$, and $\mc{C}_i$ for $i \in \mathbb{N}$ with each $\mc{C}_i$ satisfying
		\[ i \in S \Longrightarrow \mc{C}_i \cong \mc{A} \]
		and
		\[ i \notin S \Longrightarrow \mc{C}_i \cong \mc{B}.\]
		Let $(W^{i,k})_{i,k \in \mathbb{N}}$ be a uniform list of c.e.\ sets such that
		\[ i \in S \Longleftrightarrow \exists j \;\; \text{$W^{i,j}$ is infinite}\]
		and such that, when $i \in S$, there is a unique such $j$. 	We use the standard convention that at most one element enters a c.e.\ set at each stage.
		As we describe the elements of these structures $\mc{B}$, $\mc{A}$, and $\mc{C}_i$, each element will have infinitely many copies with the same labels as it; we will generally not mention this explicitly in the construction, but when we add some new element, we add infinitely many copies of it, and when we add a new label to an element, we add that label to all copies of that element.
		
		$\mc{A}$ will consist of $\mc{B}$ together with an additional element $a$ (and hence infinitely many copies of $a$). Each $\mc{C}_i$ will consist of $\mc{B}$ together with infinitely many new elements $c_{i,j}$ (and of course infinitely many copies of these). Whenever we add a new element to $\mc{B}$, we implicitly also add it to $\mc{A}$ and the $\mc{C}_i$. We will never add any further new elements to $\mc{A}$ or to some $\mc{C}_i$ without adding it to $\mc{B}$, though we may add new labels.	
		
		\medskip
		
		\noindent \textit{Construction.}
		
		\smallskip
		
		\noindent \textit{Stage 0.} $\mc{B}$ consists of infinitely many killed elements. For each $i,j \in \mathbb{N}$, $\mc{B}$ will also have an element $\hat{c}_{i,j,0}$ which has the same labels $\ell_{i,j,k}$ for $k \in \mathbb{N}$. $\mc{A}$ will contain all of the elements of $\mc{B}$ together with a distinguished element $a$ which begins with no labels. For each $i$, $\mc{C}_{i}$ will contain all of the elements of $\mc{B}$ together with an additional element $c_{i,j}$ for each $j$, with labels $\ell_{i,j,k}$ for $k \in \mathbb{N}$ (so that $\hat{c}_{i,j,0}$ has the same labels as $c_{i,j}$).
		
		\smallskip
		
		\noindent \textit{Stage $s+1$.} For each $i'$ and $j$ with $|W^{i',j}_{s+1}| > |W^{i',j}_s|$, letting $k = |W^{i',j}_s|$, do the following:
		\begin{enumerate}
			\item Put the label $\ell_{i',j,k}$ on all elements of $\mc{A}$, $\mc{B}$, and each $\mc{C}_i$.
			\item Kill $\hat{c}_{i',j,k}$ (which is in $\mc{B}$ and hence also in $\mc{A}$ and the $\mc{C}_i$).
			\item Create a new element $\hat{c}_{i',j,k+1}$ (in $\mc{B}$, and hence also in $\mc{A}$ and the $\mc{C}_{i}$). Give $\hat{c}_{i',j,k+1}$ the same labels as $c_{i',j}$ has in $\mc{C}_{i'}$ at this stage.
		\end{enumerate} 
		We use the standard convention that at most one element enters a c.e.\ set at each stage.
		
		\smallskip
		
		\noindent \textit{End construction.}
		
		\medskip
		
		\noindent It is not hard to see that $\mc{B}$ will consist of the following elements:
		\begin{itemize}
			\item Infinitely many killed elements with the label $\lkill$ and also all other labels.
			\item For each $i,j$ with $k = |W^{i,j}| < \infty$, infinitely many elements $\hat{c}_{i,j,k}$ with the labels $\ell_{i,j,k'}$ for $k' \in \mathbb{N}$ and also the labels $\ell_{i',j',k'}$ for $k' < |W^{i',j'}|$.
		\end{itemize}
		$\mc{A}$ will consist of these elements, and in addition:
		\begin{itemize}
			\item Infinitely many elements $a$ with the labels $\ell_{i,j,k}$ for $k < |W^{i,j}|$.
		\end{itemize}
		$\mc{C}_i$ will consist of the elements of $\mc{B}$, and in addition:
		\begin{itemize}
			\item For each $j$, infinitely many elements $c_{i,j}$ with the labels $\ell_{i,j,k}$ for $k \in \mathbb{N}$ and also the labels $\ell_{i',j',k'}$ for $k' < |W^{i',j'}|$.
		\end{itemize}
		The following four claims complete the proof.
		
		\begin{claim}
			If $i \notin S$ then $\mc{C}_i \cong \mc{B}$.
		\end{claim}
		\begin{proof}
			For each $j$, $W^{i,j}$ is finite, say of size $k$, and so the elements $c_{i,j}$ in $\mc{C}_i$ and $\hat{c}_{i,j,k}$ in $\mc{B} \subseteq \mc{A}$ have exactly the same labels. Since every element of $\mc{C}_i$ has the same labels as some element of $\mc{B}$, and since $\mc{C}_i$ contains $\mc{B}$, and there are infinitely many copies of each element, $\mc{C}_i \cong \mc{B}$.
		\end{proof}
		
		\begin{claim}
			If $i \in S$ then $\mc{C}_i \cong \mc{A}$.
		\end{claim}
		\begin{proof}
			Since $i \in S$, there is a unique $j$ such that $W^{i,j}$ is infinite. Then the elements $c_{i,j}$ in $\mc{C}_i$ and $a$ in $\mc{A}$ have exactly the same labels. For each $j' \neq j$, $W^{i,j'}$ is finite, say of size $k$, and so the elements $c_{i,j'}$ in $\mc{C}_i$ and $\hat{c}_{i,j',k}$ in $\mc{B} \subseteq \mc{A}$ have exactly the same labels. Thus $\mc{C}_i \cong \mc{A}$.
		\end{proof}
		
		\begin{claim}
			There is a $\Sigma_2$ sentence $\varphi$ such that $\mc{A} \models \varphi$ and $\mc{B} \nmodels \varphi$.
		\end{claim}
		\begin{proof}
			Let $N$ be the set of labels which do not hold of $a$ in $\mc{A}$. Let $\varphi = \exists x \bigdoublewedge_{\ell \in N} \neg \ell(x)$, that is, $\varphi$ says that there is an $x$ such that any label holding of $x$ also holds of $a$. We must argue that every element of $\mc{B}$ has a label that $a$ does not. Since $a$ does not have the label $\lkill$, we can restrict our attention to the elements $\hat{c}_{i,j,k}$ of $\mc{B}$ with $k = |W^{i,j}| < \infty$. Then $a$ does not have the label $\ell_{i,j,k}$, but this label holds of $\hat{c}_{i,j,k}$.
		\end{proof}
		
		\begin{claim}
			$\mc{A}$ and $\mc{B}$ have $\Pi_{3}$ Scott sentences.
		\end{claim}
		\begin{proof}
			The structures are simple enough that the automorphism orbits reduce to the orbits of singletons. Two elements are in the same automorphism orbit if and only if they have the same labels. We will show that for each element $x$ of $\mc{A}$ or $\mc{B}$ there is a $\Sigma_2$ formula such that whenever that formula holds of some $y$ in $\mc{A}$ or $\mc{B}$, $x$ and $y$ have the same labels; this will be enough to see that $\mc{A}$ and $\mc{B}$ have $\Pi_3$ Scott sentences. For the killed elements, this formula simply says that $\lkill$ holds of that element. For $\hat{c}_{i,j,k}$ with $k = |W^{i,j}|$, it is that $\ell_{i,j,k}$ holds and $\lkill$ does not. For $a$, it is the formula $\bigdoublewedge_{\ell \in N} \neg \ell(x)$ where $N$ is the set of labels which do not hold of $a$ in $\mc{A}$. The label $\lkill$ is in $N$, separating $a$ from each killed element, and for each $i,j$ with $k = |W^{i,j}| < \infty$, there is a label $\ell_{i,j,k}$ in $N$ holding of the elements $\hat{c}_{i,j,k}$ but not of $a$. This distinguishes $a$ from all of the elements of $\mc{B}$, which make up the other elements of $\mc{A}$.
		\end{proof}
		
		The claims complete the proof of the lemma.
	\end{proof}
	
	Now we give the inductive step of the proof of Theorem \ref{thm:jump-inversion-structures}, completing its proof.
	
	\begin{proof}[Proof of Theorem \ref{thm:jump-inversion-structures} for $n+1$ given Theorem \ref{thm:jump-inversion-structures} for $n$]
		By Lemma \ref{lem:n1} relativized to $\mathbf{0}^{(2n)}$, we get the following. Let $S \subseteq \mathbb{N}$ be a complete $\Sigma^0_{2n+3}$ set, which is a complete $\Sigma^0_3$ relative to $\mathbf{0}^{(2n)}$ set. There are $\mathbf{0}^{(2n)}$-computable structures $\mc{A}$ and $\mc{B}$ with $\Pi_{3}$ Scott sentences and a sequence of $\mathbf{0}^{(2n)}$-computable structures $\mc{C}_i$ such that
		\[ i \in S \Longrightarrow \mc{C}_i \cong \mc{A} \]
		and
		\[ i \notin S \Longrightarrow \mc{C}_i \cong \mc{B}.\]
		Moreover, there is a $\Sigma_2$ sentence $\varphi$ such that $\mc{A} \models \varphi$ and $\mc{B} \nmodels \varphi$.
		
		Now let $\mc{A}^* = \Inv{n}{\mc{A}}$, $\mc{B}^* = \Inv{n}{\mc{B}}$, and $\mc{C}_i^* = \Inv{n}{\mc{C}_i}$. We use the properties of Theorem \ref{thm:jump-inv}. Then $\mc{A}^*$ and $\mc{B}^*$ have computable copies with $\Pi_{n+3}$ Scott sentences. $\mc{C}_i^*$ is a computable sequence of computable structures such that
		\[ i \in S \Longrightarrow \mc{C}^*_i \cong \mc{A}^* \]
		and
		\[ i \notin S \Longrightarrow \mc{C}^*_i \cong \mc{B}^*.\]
		Finally, we obtain from Theorem \ref{thm:jump-inv} a $\Sigma_{n+2}$ sentence $\varphi^*$ such that $\mc{A}^* \models \varphi^*$ and $\mc{B}^* \nmodels \varphi^*$.
	\end{proof}
	
	\section{Applications}\label{sec:applications}
	
	\subsection{Computable Scott sentences}
	
	In this section, we will prove that when a computable structure has a $\Pi_n$ Scott sentence, the best we can say is that it has a $\mathbf{0}^{(n)}$-computable $\Pi_n$ Scott sentence, and a computable $\Pi_{2n}$ Scott sentence. Recall that Alvir, Csima, and Harrison-Trainor \cite{AlvirCsimaHT} proved that there is a computable structure with a $\Pi_2$ Scott sentence but no computable $\Sigma_{4}$ Scott sentence. We want to show that there is a computable structure with a $\Pi_n$ Scott sentence but no computable $\Sigma_{2n}$ Scott sentence. Morally speaking, we prove the theorem above by relativizing that fact and applying jump inversion. However, this argument does not quite work. Relativizing to $\mathbf{0}^{(2n-4)}$, there is a $\mathbf{0}^{(2n-4)}$-computable structure $\mc{A}$ with a $\Pi_2$ Scott sentence but no $\mathbf{0}^{(2n-4)}$-computable $\Sigma_{4}$ Scott sentence. Then $\Inv{(n-2)}{\mc{A}}$ has a computable copy and a $\Pi_{n}$ Scott sentence; however, it is not clear that it has no computable $\Sigma_{2n}$ Scott sentence. Instead, we prove the following stronger fact.
	
	\begin{theorem}\label{thm:no-comp-scott-sentence-more}
		Fix $n \geq 2$. Let $S \subseteq \mathbb{N}$ be a complete $\Pi^0_{2n}$ set. There is a computable structure $\mc{A}$ with a $\Pi_{n}$ Scott sentence and a uniformly computable sequence of structures $\mc{C}_i$ such that
		\[ i \in S \Longrightarrow \mc{C}_i \cong \mc{A} \]
		and
		\[ i \notin S \Longrightarrow \mc{C}_i \ncong \mc{A}.\]
		In particular, $\mc{A}$ has no computable $\Sigma_{2n}$ Scott sentence.
	\end{theorem}
	
	The case $n = 2$ of this theorem is proved very similarly to the corresponding fact from \cite{AlvirCsimaHT}. The proof is quite technical and contains only minor differences and so we omit it.
	Given this fact, we can apply the jump inversion as follows.
	
	\begin{proof}[Proof of the general case given the case $n = 2$]
		Fix $n \geq 3$. Given a complete $\Pi^0_{2n}$ set $S$, it is a complete $\Pi^0_4$ set relative to $\mathbf{0}^{(2n-4)}$. Relativizing the case $n=2$ to $\mathbf{0}^{(2n-4)}$, there is a $\mathbf{0}^{(2n-4)}$-computable structure $\mc{A}$ with a $\Pi_2$ Scott sentence, and a $\mathbf{0}^{(2n-4)}$-computable sequence of structures $\mc{C}_i$, such that
		\[ i \in S \Longrightarrow \mc{C}_i \cong \mc{A} \]
		and
		\[ i \notin S \Longrightarrow \mc{C}_i \ncong \mc{A}.\]
		Let $\mc{A}^* = \Inv{(n-2)}{\mc{A}}$ and $\mc{C}_i^* =  \Inv{(n-2)}{\mc{C}_i}$ be unfriendly $(n-2)$-jump inversions. We use Theorem \ref{thm:jump-inv}: $\mc{A}^*$ has a computable copy and there is a uniformly computable sequence of computable copies of $\mc{C}_i^*$. $\mc{A}^*$ has a $\Pi_{n}$ Scott sentence and
		\[ i \in S \Longrightarrow \mc{C}_i^* \cong \mc{A}^* \]
		and
		\[ i \notin S \Longrightarrow \mc{C}_i^* \ncong \mc{A}^*.\]
		This completes the argument.
	\end{proof}
	
	\subsection{Back-and-forth types}
	
	Recall that in \cite{ChenGonzalezHT}, Chen, Gonzalez, and Harrison-Trainor showed that for each fixed $\mc{A}$ the set 
	\[ \{ \mc{B} \in \Mod(\mc{L}) : \mc{A} \leq_n \mc{B}\}\]
	is always (boldface) $\bfPi^0_{n+2}$. They asked, if $\mc{A}$ is computable, what the lightface complexity of this set is. The upper bound is $\Pi^0_{2n}$. For the structures constructed in Theorem \ref{thm:no-comp-scott-sentence-more}, and as verified in the proof of that theorem, this is best possible.
	
	\backandforthtypes*

	\subsection{Separating structures}
	
	Recall that Ash and Knight \cite{AshKnight} showed that if $\mc{A}$ and $\mc{B}$ are computable structures, $n$-friendly with computable existential diagrams, and $\mc{B} \nleq_n \mc{A}$, then there is a computable $\Pi_n$ sentence $\varphi$ such that $\mc{B} \models \varphi$ and $\mc{A} \nmodels \varphi$. If we drop the assumption that $\mc{A}$ and $\mc{B}$ are $n$-friendly with computable existential diagrams, then there is a computable $\Pi_{2n-1}$ sentence $\varphi$.\footnote{The proof is similar to that of Theorem \ref{thm:separating-structures} where we showed that if $\mc{A}$ and $\mc{B}$ are computable structures, and $\mc{A} \nleq_n \mc{B}$, then there is a $\mathbf{0}^{(n-1)}$-computable $\Pi_n$ sentence $\varphi$ such that $\mc{A} \models \varphi$ and $\mc{B} \nmodels \varphi$.} We get the following example showing that this is best possible.
	
	\separating*
	
	\begin{proof}
		Let $S$ be a $\Sigma^0_{2n-1}$-complete set. Fix, as in Theorem \ref{thm:jump-inversion-structures}, computable structures $\mc{A} \ncong \mc{B}$ and a sequence of computable structures $\mc{C}_i$ such that
		\[ i \in S \Longrightarrow \mc{C}_i \cong \mc{A} \]
		and
		\[ i \notin S \Longrightarrow \mc{C}_i \cong \mc{B}.\]
		Moreover, $\mc{B} \nleq_{n} \mc{A}$. If there was a computable $\Pi_{2n-1}$ sentence $\varphi$ such that $\mc{A} \models \varphi$ and $\mc{B} \nmodels \varphi$, then
		\[ S = \{ i : \mc{C}_i \models \varphi\}\]
		would be $\Pi^0_{2n-1}$. Since it is $\Sigma^0_{2n-1}$-complete, there is no such $\varphi$.
	\end{proof}
	
	\subsection{Maximally unfriendly structures}
	
	We have noted previously that if $\mc{A}$ is a computable structure the back-and-forth relations $\leq_n$ on $\mc{A}$ are $\Pi^0_{2n}$. We show that this is best possible, that is, there is a maximally unfriendly structure.
	
	\maxunfriendly*
	
	\begin{proof}
		For $n = 1$ a straightforward argument suffices. Consider the case $n \geq 2$. Let $R \subseteq \mathbb{N}$ be a $\Pi^0_{2n}$-complete set.  Write
		\[ k \in R \Longleftrightarrow \forall \ell \; (k,\ell) \in S\]
		where $S$ is $\Sigma^0_{2n-1}$. Fix, as in Theorem \ref{thm:jump-inversion-structures}, computable structures $\mc{A} \ncong \mc{B}$ and a sequence of computable structures $\mc{C}_{k,\ell}$ such that
		\[ (k,\ell) \in S \Longrightarrow \mc{C}_{k,\ell} \cong \mc{A} \]
		and
		\[ (k,\ell) \notin S \Longrightarrow \mc{C}_{k,\ell} \cong \mc{B}.\]
		Moreover, $\mc{B} \nleq_{n} \mc{A}$.
		
		$\mc{M}$ will have a unary relation $U$, binary relations $D_\ell$, and the symbols of the language of $\mc{A}$ and $\mc{B}$. For each element $a$ satisfying $U$, we will attach to $a$ a set of elements $D_\ell(a) = \{ b : D_\ell(a,b)\}$. None of these elements will satisfy $U$, and the sets $D_\ell(a)$ will partition the non-$U$ elements. Each set $D_\ell(a)$ will be the domain of a structure isomorphic to $\mc{A}$ or $\mc{B}$. Essentially, one should think of $\mc{M}$ as consisting of infinitely many elements $a$, each of which has attached to it an ordered sequence of structures $D_\ell(a)$ each of which is isomorphic to $\mc{A}$ or $\mc{B}$.
		
		$\mc{M}$ will have elements $a_k$, each of which has as the $\ell$th structure $D_\ell(a_k)$ attached to it a copy of $\mc{C}_{k,\ell}$. If $k \in R$, then each of these structures is isomorphic to $\mc{A}$, while if $k \notin R$, then at least one of these structures is isomorphic to $\mc{B}$. $\mc{M}$ will also have one additional element $a^*$ with each structure attached to it being a copy of $\mc{A}$.
		
		Standard arguments with back-and-forth relations show that $(\mc{M},a^*) \geq_n (\mc{M},a_k)$ if and only if $k \in R$: If $k \in R$, then $a_k$ and $a^*$ have the same structures attached, all copies of $\mc{A}$, while if $k \notin R$, then for some $\ell$ the $\ell$th structure attached to $a_k$ is isomorphic to $\mc{B}$, but the $\ell$th structure attached to $a^*$ is isomorphic to $\mc{A}$, and $\mc{A} \ngeq_n \mc{B}$. This yields a $1$-reduction from the $\Pi^0_{2n}$-complete set $R$ to the $n$-back-and-forth relation on $\mc{M}$ which is thus $\Pi^0_{2n}$-complete and computes $\mathbf{0}^{(2n)}$.
	\end{proof}
	
	\section{Non-arithmetic is a degree spectrum}\label{sec:non-arithmetic-proof}
	
	In this section we prove the main theorem of the paper.
	
	\nonarithmetic*
	\begin{proof}
		Our structure will consist, for each $n$, of a sort which diagonalizes against the $\mathbf{0}^{(n)}$-computable structures. For each $n$, we verify the following lemma which is of interest on its own. It says that for each $n$ there is a structure whose degree spectrum is contained in the non-$\Delta^0_{n+1}$ degrees, and contains all of the non-$\Delta^0_{2n+1}$ degrees.
		
		\begin{lemma}
			There is a structure $\mc{M}_n$ which has no $\mathbf{0}^{(n)}$-computable copy, but for each set $X$ with $X \nleq_T \mathbf{0}^{(2n)}$, there is an $X$-computable copy of $\mc{M}_n$.
		\end{lemma}
		
		With a little extra work we could make this uniform in $X$, but in the construction we give there will be two uniform procedures for building $\mc{M}_n$ from $X$, one which works if $X$ is not $\Sigma^0_{2n+1}$, and the other working if $\overline{X}$ is not $\Sigma^0_{2n+1}$. If $X$ is non-arithmetic, then it is not $\Sigma^0_{2n+1}$ for any $n$, and so there will be a uniform construction of $\mc{M}_n$ from $X$.
		
		Letting $\mc{M}$ be the structure whose $n$th sort consists of a copy of $\mc{M}_n$, the degree spectrum of $\mc{M}$ is the non-arithmetic degrees. For each arithmetic set $X$, $X \leq_T \mathbf{0}^{(n)}$ for some $n$, and so $X$ cannot even compute a copy of the $n$th sort of $\mc{M}$. On the other hand, for each non-arithmetic set $X$, for each $n$ we have $X \nleq_T \mathbf{0}^{(2n)}$; thus, $X$ can compute a copy of $\mc{M}_n$. Moreover, since $X$ is not $\Sigma^0_{2n+1}$, $X$ can uniformly construct a copy of $\mc{M}_n$, and hence $X$ can compute a copy of $\mc{M}$. Thus it remains to prove the lemma.
		
		\begin{proof}[Proof of lemma]
			Fix $n$. We will construct $\mc{M} = \mc{M}_n$ as in the claim. $\mc{M}$ will essentially have the form of an unfriendly jump inversion, though we will have to build $\mc{M}$ directly rather than defining it as the jump inversion of some other structure and applying Theorem \ref{thm:jump-inv}. (Technically, it will be constructed using a Marker extension rather than a jump inversion.)
			
			Let $\mc{A} \ncong \mc{B}$ be computable $\mc{L}$-structures as in Theorem \ref{thm:jump-inversion-structures}, i.e., for $S$ a complete $\Sigma^0_{2n+1}$ set there is a computable sequence of structures $\mc{C}_i$ such that
			\[ i \in S \Longleftrightarrow \mc{C}_i \cong \mc{A}\]
			and
			\[ i \notin S \Longleftrightarrow \mc{C}_i \cong \mc{B}.\]
			Moreover, we have a $\mathbf{0}^{(n)}$-computable $\Sigma_{n+1}$ sentence $\varphi$ such that $\mc{A} \models \varphi$ and $\mc{B} \nmodels \varphi$.
			
			Consider the language $\{U_k : k \in \mathbb{N}\} \cup \{R_\ell : \ell \in \mathbb{N}\}$ where all of these are unary relations, and we think of the $R_\ell$ as being c.e.\ relations. This is (one choice for) the language in which the Slaman-Wehner structures are defined. The $U_k$ are sorts, and we think of an element $a$ as c.e.-coding the set $\{ \ell : R_\ell(a)\}$. The standard idea for the construction is to have the elements of the $k$th sort code exactly the family $\{F \subseteq \mathbb{N} : \text{$F$ is finite and $F \neq W_k$}\}$. The $R_\ell$ are made into c.e.\ relations by, e.g., replacing them in the structure by attaching a loop of length $\ell+3$ to $a$ to signify that $R_\ell(a)$. It is easier to think of the $R_\ell$ as being in the language but being c.e.\ in a computable copy.
			
			In constructing $\mc{M}$, rather than making the relations $R_\ell$ c.e., we will apply a variant of the unfriendly Marker extensions to the relations $R_\ell$ while leaving the $U_k$ alone. The language will consist of the unary relations $U_k$ as well as additional binary relations $D_{\ell}$ and $\overline{D}_{\ell}$, an equivalence relation $E$, and the symbols from $\mc{L}$ the language of $\mc{A}$ and $\mc{B}$. The $U_k$ will still be sorts, but we now have elements not satisfying any of the $U_k$. For each $x \in U_k$, the relations $D_{\ell}$ and $\overline{D}_{\ell}$ will associate with $x$ two infinite sets
			\[ D_{\ell}(x) = \{ y : D_{\ell}(x,y)\} \text{ and } \overline{D}_{\ell}(x) = \{y : \overline{D}_{\ell}(x,y)\}.\]
			These sets will partition the elements which are not in any of the sorts $U_k$. $E$ will be an equivalence relation on all of the non-$U$ elements and will divide each of these sets $D_{\ell}(x)$ and $\overline{D}_\ell(x)$ into infinitely many infinite equivalence classes. Each equivalence class will be split across both $D_{\ell}(x)$ and $\overline{D}_{\ell}(x)$ for some $x$. Each equivalence class, when restricted to $D_{\ell}(x)$ or $\overline{D}_{\ell}(x)$, will be the domain of an $\mc{L}$-structure. Thus each of these $\mc{L}$-structures in $D_{\ell}(x)$ will have an associated $\mc{L}$-structure in $\overline{D}_{\ell}(x)$ in the same equivalence class, which we call its dual. The elements of these $\mc{L}$-structures will be elements outside of the sorts $U_k$, and these structures will all be disjoint from each other. Among each $\mc{L}$-structure and its dual, one will be isomorphic to $\mc{A}$ and the other to $\mc{B}$. See Figure \ref{fig} for a depiction of the elements associated to some $x$.
			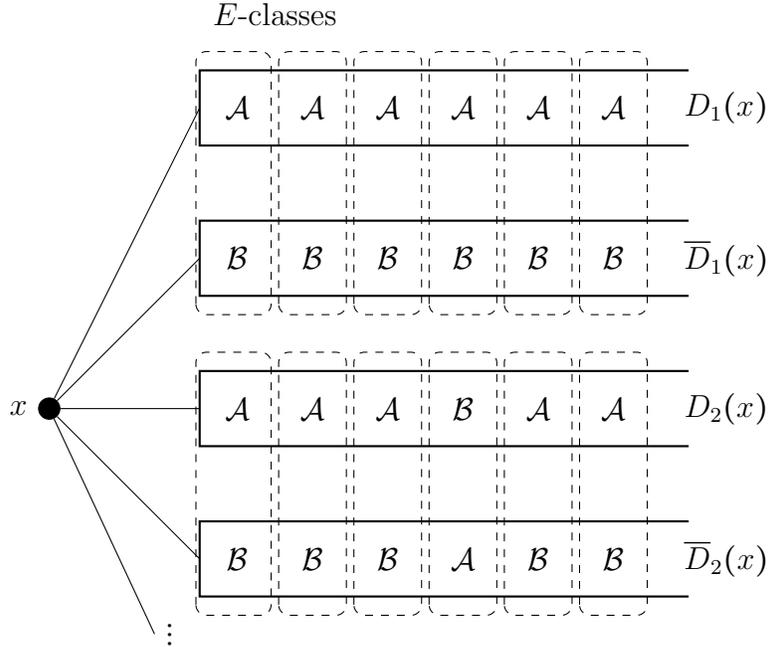
\begin{figure}[ht]
				\[ \begin{tikzpicture}
					\node (v1) at (0,0)     [circle, fill=black,draw=black, inner sep=0.1cm, label=left:$x$] {};
					
					\node at (9,4) [] {$D_1(x)$};
					\node at (9,2) [] {$\overline{D}_1(x)$};
					\node at (9,0) [] {$D_2(x)$};
					\node at (9,-2) [] {$\overline{D}_2(x)$};
					
					\node at (3,5.25) [] {$E$-classes};
					
					\node at (2.5,4) [] {$\mc{A}$};
					\node at (2.5,2) [] {$\mc{B}$};
					\node at (3.5,4) [] {$\mc{A}$};
					\node at (3.5,2) [] {$\mc{B}$};
					\node at (4.5,4) [] {$\mc{A}$};
					\node at (4.5,2) [] {$\mc{B}$};
					\node at (5.5,4) [] {$\mc{A}$};
					\node at (5.5,2) [] {$\mc{B}$};
					\node at (6.5,4) [] {$\mc{A}$};
					\node at (6.5,2) [] {$\mc{B}$};
					\node at (7.5,4) [] {$\mc{A}$};
					\node at (7.5,2) [] {$\mc{B}$};
					
					\node at (2.5,0) [] {$\mc{A}$};
					\node at (2.5,-2) [] {$\mc{B}$};
					\node at (3.5,0) [] {$\mc{A}$};
					\node at (3.5,-2) [] {$\mc{B}$};
					\node at (4.5,0) [] {$\mc{A}$};
					\node at (4.5,-2) [] {$\mc{B}$};
					\node at (5.5,0) [] {$\mc{B}$};
					\node at (5.5,-2) [] {$\mc{A}$};
					\node at (6.5,0) [] {$\mc{A}$};
					\node at (6.5,-2) [] {$\mc{B}$};
					\node at (7.5,0) [] {$\mc{A}$};
					\node at (7.5,-2) [] {$\mc{B}$};
					
					\node (dots) at (1.6,-3)     [] {$\vdots$};
					
					\draw[thick] (8.5,4.5) -- (2,4.5) -- (2,3.5) -- (8.5,3.5);
					
					\draw[thick] (8.5,2.5) -- (2,2.5) -- (2,1.5) -- (8.5,1.5);
					
					\draw [draw=black,rounded corners,dashed] (1.95,1.25) rectangle (2.95,4.75);
					\draw [draw=black,rounded corners,dashed] (3.05,1.25) rectangle (3.95,4.75);
					\draw [draw=black,rounded corners,dashed] (4.05,1.25) rectangle (4.95,4.75);
					\draw [draw=black,rounded corners,dashed] (5.05,1.25) rectangle (5.95,4.75);
					\draw [draw=black,rounded corners,dashed] (6.05,1.25) rectangle (6.95,4.75);
					\draw [draw=black,rounded corners,dashed] (7.05,1.25) rectangle (7.95,4.75);
					
					\draw[thick] (8.5,0.5) -- (2,0.5) -- (2,-0.5) -- (8.5,-0.5);
					
					\draw[thick] (8.5,-2.5) -- (2,-2.5) -- (2,-1.5) -- (8.5,-1.5);
					
					\draw [draw=black,rounded corners,dashed] (1.95,-2.75) rectangle (2.95,0.75);
					\draw [draw=black,rounded corners,dashed] (3.05,-2.75) rectangle (3.95,0.75);
					\draw [draw=black,rounded corners,dashed] (4.05,-2.75) rectangle (4.95,0.75);
					\draw [draw=black,rounded corners,dashed] (5.05,-2.75) rectangle (5.95,0.75);
					\draw [draw=black,rounded corners,dashed] (6.05,-2.75) rectangle (6.95,0.75);
					\draw [draw=black,rounded corners,dashed] (7.05,-2.75) rectangle (7.95,0.75);
					
					\draw   (v1) -- (2,4);
					\draw   (v1) -- (2,2);
					\draw   (v1) -- (2,0);
					\draw   (v1) -- (2,-2);
					
					\draw   (v1) -- (1.4,-3);
				\end{tikzpicture}\]
				\caption{An element $x \in U_k$ and all of its associated elements. The element $x$ shown here satisfies $\neg R_1(x)$ because each $D_1(x)$-structure is isomorphic to $\mc{A}$ and each $\overline{D}_1(x)$-structure is isomorphic to $\mc{B}$. It satisfies $R_2(x)$ as witnessed by the fourth equivalence class in $D_2(x)$ and $\overline{D}_2(x)$.}\label{fig}
			\end{figure}

			For each $x$, either all of the $\mc{L}$-structures in $D_{\ell}(x)$ will be isomorphic to $\mc{A}$ (and hence all of the $\mc{L}$-structures in $\overline{D}_{\ell}(x)$ will be isomorphic to $\mc{B}$) or one will be isomorphic to $\mc{B}$ and the others all isomorphic to $\mc{A}$ (and hence one $\mc{L}$-structure in $\overline{D}_{\ell}(x)$ will be isomorphic to $\mc{A}$ and the others to $\mc{B}$). Thus we have the definable relation $R_\ell(x)$ which holds of $x$ if and only if one of the $\mc{L}$-structures of $D_\ell(x)$ is isomorphic to $\mc{B}$ if and only if one of the $\mc{L}$-structures of $\overline{D}_\ell(x)$ is isomorphic to $\mc{A}$. Note that $R_\ell$ is definable by a $\mathbf{0}^{(n)}$-computable $\Sigma_{n+1}$ formula.\footnote{The reader might wonder why we cannot just attach to each $x$ for each $m$ an $\mc{L}$-structure isomorphic to either $\mc{A}$ or $\mc{B}$, with $R_m(x)$ if and only if the structure is isomorphic to $\mc{B}$. We would still have that $R_m$ is $\Sigma^0_{2n+1}$ in any computable copy. However in non-computable copies the complexities would not be correct.}
			
			To a structure $\mc{N}$ of the above type, we can assign certain families of sets coded by the structure. Each point $x$ in one of the sorts $U_k$ codes a set
			\begin{align*}
				F(x) &= \{\ell : R_\ell(x)\} \\ &= \{ \ell : \text{one of the structures attached to $x$ via $D_\ell$ is $\cong \mc{B}$}\} \\ &= \{ \ell : \text{one of the structures attached to $x$ via $\overline{D}_\ell$ is $\cong \mc{A}$}\}
			\end{align*}
			Then the $k$th sort of $\mc{N}$ codes the family of sets 
			\[ \mc{F}_k = \{ F(x) : \text{$x$ is in the $k$th sort $U_k$ of $\mc{N}$}\}.\]
			We will consider structures which have infinite multiplicity, i.e., for each set $F \in \mc{F}_k$, there are infinitely many $x \in U_k$ with $F(x) = F$.
			Thus $\mc{N}$ codes the sequence of families of sets $(\mc{F}_k)_{k \in \mathbb{N}}$. Moreover, to each sequence of families $\mc{F} = (\mc{F}_k)_{k \in \mathbb{N}}$ we can construct a structure coding $\mc{F}$, with $\mc{F}_k$ coded in the $k$th sort with infinite multiplicity.
			
			Let $\mc{M}$ be the structure whose $k$th sort corresponds to the family
			\[ \{ F : \text{$F$ finite and $F \neq W^{\mathbf{0}^{(2n)}}_k$}\}\]
			where $W^{\mathbf{0}^{{(2n)}}}_k$ is the $k$th $\Sigma^0_{2n+1}$ set.
			
			\begin{claim}\label{cl:one}
				There is no $\mathbf{0}^{(n)}$-computable copy of $\mc{M}$.
			\end{claim}
			\begin{proof}
				In a $\mathbf{0}^{(n)}$-computable copy of $\mc{M}$, let $a_k$ be the first element of the $k$th sort, and let $V_k = F(a_k)$. Then $V_k$ is $\Sigma^0_{2n+1}$ as $\ell \in V_k$ if and only if one of the structures attached to $a_k$ via $\overline{D}_{\ell}$ is isomorphic to $\mc{A}$ (and we have a $\mathbf{0}^{(n)}$-computable $\Sigma_{n+1}$ formula distinguishing $\mc{A}$ from $\mc{B}$). Moreover, we can find a $\Sigma^0_{2n+1}$ index for $V_k$ uniformly in $k$, and hence a c.e.-relative-to-$\mathbf{0}^{(2n)}$ index for $V_k$. Then, by the recursion theorem, for some $k$, we have $V_k = W_k^{\mathbf{0}^{(2n)}}$. But then this was not a copy of $\mc{M}$ after all.
			\end{proof}
			
			In the second claim we make use of the idea of what we call \textit{$S$-indices} for $\Sigma^0_{2n+1}$ facts, where $S$ is the $\Sigma^0_{2n+1}$-complete set fixed above. By an $S$-index we mean some $I \in \mathbb{N}$. $I$ represents \textsf{true} or $1$ if $I \in S$, in which case we write $I \equiv 1$, and represents \textsf{false} or $0$ if $I \notin S$, in which case we write $I \equiv 0$. Because $S$ is $\Sigma^0_{2n+1}$-complete, any $\Sigma^0_{2n+1}$ fact has an $S$-index, and we can generally compute such an index in reasonable situations. For example, given some positive Boolean combination of facts represented by $S$-indices, we can compute an $S$-index for the result, e.g., given $S$-indices $I_1$, $I_2$, and $I_3$ we can compute an $S$-index $J$ representing ``($I_1$ and $I_2$) or $I_3$'', that is,
			\[ J \in S \Longleftrightarrow \text{$I_1$ and $I_2$ are both in $S$, or $I_3$ is in $S$.}\]
			We can also represent $\Delta^0_{2n+1}$ facts by giving a pair of $S$-indices $I$ and $\overline{I}$ with $I \equiv 1 \Longleftrightarrow \overline{I} \equiv 0$.
			
			The importance of $S$-indices is that given an $S$-index $I$, \[\mc{C}_I \cong \mc{A} \Longleftrightarrow I \equiv 1 \text{ and }\mc{C}_I \cong \mc{B} \Longleftrightarrow I \equiv 0.\]
			The $\mc{C}_I$ are computable, but in a construction of some larger structure we might non-computably build some sequence of $S$-indices $I_1,I_2,\ldots$ and then incorporate the structures $\mc{C}_{I_1},\mc{C}_{I_2},\ldots$ into the larger structure.
			
			\begin{claim}\label{cl:two}
				Suppose that $X$ is not $\Sigma^0_{2n+1}$. Then there is (uniformly in $X$) an $X$-computable copy $\mc{N}$ of $\mc{M}$.
			\end{claim}
			\begin{proof}
				Fix $k$. We want to build (uniformly in $X$ and $k$) a copy $\mc{N}_k^X$ of the $k$th sort coding
				\[ \{ F : \text{$F$ finite and $F \neq W^{\mathbf{0}^{{(2n)}}}_k$}\}.\]
				If we can do this uniformly, then we can compute uniformly in $X$ a copy of $\mc{M}$.
				Our elements will be of the form $a_{F,s}$ for each finite set $F$ and $s \in \mathbb{N}$. We will build our structure to be $X$-computable in the following way. For each $a_{F,s}$, $\ell$, and $t$, we will have, $X$-computably, two $S$-indices $I(F,s,\ell,t)$ and $\overline{I}(F,s,\ell,t)$ which are complementary, i.e., $I(F,s,\ell,t) \nequiv \overline{I}(F,s,\ell,t)$. $I(F,s,\ell,t)$ and $\overline{I}(F,s,\ell,t)$ describe which structures should be attached to $a_{F,s}$ in $\mc{N}_k^X$. 
				
				The $t$th structure attached to $a_{F,s}$ in $D_\ell(a_{F,s})$ will be:
				\begin{itemize}
					\item $\mc{A}$, if $I(F,s,\ell,t) \equiv 0$ (and $\overline{I}(F,s,\ell,t) \equiv 1$), and
					\item $\mc{B}$, if $I(F,s,\ell,t) \equiv 1$ (and $\overline{I}(F,s,\ell,t) \equiv 0$)
				\end{itemize}
				The $t$th structure attached to $a_{F,s}$ in $\overline{D}_\ell(a_{F,s})$ will be:
				\begin{itemize}
					\item $\mc{B}$, if ${I}(F,s,\ell,t) \equiv 0$ (and $\overline{I}(F,s,\ell,t) \equiv 1$), and
					\item $\mc{A}$, if ${I}(F,s,\ell,t) \equiv 1$ (and $\overline{I}(F,s,\ell,t) \equiv 0$).
				\end{itemize}
				The $t$th structures in $D_\ell(a_{F,s})$ and $\overline{D}_\ell(a_{F,s})$ will be in the same $E$-equivalence class. For each $F,s,\ell$ there will be at most one $t$ with $I(F,s,\ell,t) \equiv 1$ and $\overline{I}(F,s,\ell,t) \equiv 0$.
				
				To compute a copy of the structure $\mc{N}_k^X$ using $X$, we put as the $t$th equivalence class in $D_\ell(a_{F,s})$ a copy of $\mc{C}_{\overline{I}(F,s,\ell,t)}$ and in $\overline{D}_{\ell}(a_{F,s})$ a copy of $\mc{C}_{I(F,s,\ell,t)}$. Then we have $\mc{C}_{\overline{I}(F,s,\ell,t)} \cong \mc{A}$ and $\mc{C}_{{I}(F,s,\ell,t)} \cong \mc{B}$ if $I(F,s,\ell,t) \equiv 0$ and $\overline{I}(F,s,\ell,t) \equiv 1$, and $\mc{C}_{\overline{I}(F,s,\ell,t)} \cong \mc{B}$ and $\mc{C}_{{I}(F,s,\ell,t)} \cong \mc{A}$ if $I(F,s,\ell,t) \equiv 1$ and $\overline{I}(F,s,\ell,t) \equiv 0$, as desired. Since the sequence $\mc{C}_i$ is computable and the indices $I(F,s,\ell,t)$ and $\overline{I}(F,s,\ell,t)$ are $X$-computable, this procedure builds an $X$-computable structure. It remains to define $I(F,s,\ell,t)$ and $\overline{I}(F,s,\ell,t)$.
				
				Recall that we previously described some aspects of the structure $\mc{M}$. From this, it follows that $R_\ell(x)$ holds of $x$ if and only if one of the $\mc{L}$-structures of $D_\ell(x)$ is isomorphic to $\mc{B}$, equivalently, if and only if one of the $\mc{L}$-structures of $\overline{D}_\ell(x)$ is isomorphic to $\mc{A}$. Then we have $R_{\ell}(a_{F,s})$ if and only if there is some $t$ such that $I(F,s,\ell,t) \equiv 1$ and $\overline{I}(F,s,\ell,t) \equiv 0$. We can get approximations to this as follows.
				
				We define, for each $t$,
				\[ R_{F,s}[t] = \{ \ell : \exists t' \leq t \; I(F,s,\ell,t') \equiv 1 \} = \{ \ell : \exists t' \leq t \; \overline{I}(F,s,\ell,t') \equiv 0 \}.\]
				This is a $\Delta^0_{2n+1}$ set and for each $\ell$ we can compute $S$-indices for the facts ``$\ell \in R_{F,s}[t]$'' and ``$\ell \notin R_{F,s}[t]$'' from $S$-indices $I(F,s,\ell,t')$ and $\overline{I}(F,s,\ell,t')$. That is, given $F,s,\ell,t$ we can compute some $j,\overline{j}$ such that 
				\[ \ell \in R_{F,s}[t] \Longleftrightarrow j \in S \]
				and
				\[ \ell \notin R_{F,s}[t] \Longleftrightarrow \overline{j} \in S.\]
				Also, let $W^{\mathbf{0}^{(2n)}}_{k,t}$ be the stage $t$ approximation to $W^{\mathbf{0}^{(2n)}}_k$, i.e., $\ell \in W^{\mathbf{0}^{(2n)}}_{k,t}$ if and only if $\Phi_{k,t}^{\mathbf{0}^{(2n)}}(\ell) \downarrow$. $W^{\mathbf{0}^{(2n)}}_{k,t}$ is $\Delta^0_{2n+1}$ uniformly in $t$, and we can compute $S$-indices for ``$\ell \in W^{\mathbf{0}^{(2n)}}_{k,t}$'' and ``$\ell \notin W^{\mathbf{0}^{(2n)}}_{k,t}$''.
				
				Define $I(F,s,\ell,t)$ and $\overline{I}(F,s,\ell,t)$ as follows, inductively on $t$. For each $t$ we define them for all $F,s,\ell$ and so we can use $R_{F,s}[t]$ in the definitions at stage $t+1$. We begin by defining $I(F,s,\ell,t)$ and $\overline{I}(F,s,\ell,t)$ by represented value, i.e., whether $I(F,s,\ell,t) \equiv 0$ or $I(F,s,\ell,t)\equiv 1$, before explaining how they can be defined as $S$-indices. The idea is that we carry out something like the standard Slaman-Wehner argument, relative to $\mathbf{0}^{(2n)}$, using $X$ as an oracle.
				\begin{enumerate}
					\item For $t =0$, set
					\begin{enumerate}
						\item $I(F,s,\ell,t) \equiv 1$ and $\overline{I}(F,s,\ell,t) \equiv 0$ if $\ell \in F$, and
						\item $I(F,s,\ell,t) \equiv 0$ and $\overline{I}(F,s,\ell,t) \equiv 1$ if $\ell \notin F$.
					\end{enumerate} 
					\item For $1 \leq t \leq s$ set $I(F,s,\ell,t) \equiv 0$ and $\overline{I}(F,s,\ell,t) \equiv 1$.
					\item At stage $t+1 > s$, find the first $t+1$ elements of $X$, and let
					\begin{enumerate}
						\item $I(F,s,\ell,t+1) \equiv 1$ and $\overline{I}(F,s,\ell,t+1) \equiv 0$ if
						$R_{F,s}[t] = W_{k,t}^{\mathbf{0}^{(2n)}}$ and $\ell$ is the least element of $X - W_{k,t}^{\mathbf{0}^{(2n)}}$;
						\item $I(F,s,\ell,t+1) \equiv 0$ and $\overline{I}(F,s,\ell,t+1) \equiv 1$ otherwise.
					\end{enumerate}
				\end{enumerate}
				Note that by ``define,'' we mean that we compute using oracle $X$ the $S$-indices for $I$ and $\overline{I}$. For (1) and (2) we can simply fix $s_1 \in S$ and $s_0 \notin S$, setting, e.g., $I(F,s,\ell,0) = s_1$ if $I(F,s,\ell,0) \equiv 1$. For (3), note that the values of $I(F,s,\ell,t+1)$ and $\overline{I}(F,s,\ell,t+1)$ depend on finitely many $\Delta^0_{2n+1}$ questions that we need to ask, and we can compute which questions we must ask using $X$. It is important for this that $W_{k,t}^{\mathbf{0}^{(2n)}}$ has only elements $< t$. We have $S$-indices for each of these $\Delta^0_{2n+1}$ facts, and so we can combine these to get $S$-indices for $I(F,s,\ell,t+1)$ and $\overline{I}(F,s,\ell,t+1)$. The process of computing indices is $X$-computable.
				
				Now we must verify that this construction works. This is essentially the usual Slaman-Wehner argument. First we show that for every $F,s$ we have that
				\[ R_{F,s} = \bigcup_t R_{F,s}[t]\]
				is finite and not equal to $W_{k}^{\mathbf{0}^{(2n)}}$. Indeed, if we had $R_{F,s} = W_{k}^{\mathbf{0}^{(2n)}}$, then this would be because $R_{F,s} = W_{k}^{\mathbf{0}^{(2n)}} = F \cup X$ and $F \cup X$ is not $\Sigma^0_{2n+1}$ while $W_{k}^{\mathbf{0}^{(2n)}}$ is. On the other hand, for every finite set $F \neq W_{k}^{\mathbf{0}^{(2n)}}$, we have $R_{F,s} = F$ for large enough $s$: simply take $s$ so that for every $t \geq s$ we have $F \neq W_{k,t}^{\mathbf{0}^{(2n)}}$.
				Thus we have
				\[ \{ R_{F,s} : \text{$F$ finite, $s \in \mathbb{N}$}\} = \{ F : \text{$F$ finite and $F \neq W^{\mathbf{0}^{{(2n)}}}_k$}\}\]
				with infinite multiplicity on the left-hand-side.
				It just remains to note that $R_{F,s} = \{ \ell : R_\ell(a_{F,s})\}$ is the set coded by $a_{F,s}$ in $\mc{N}_k^X$. Thus $\mc{N}^X_k$ is isomorphic to the $k$th sort of $\mc{M}$, and hence $\mc{N}^X \cong \mc{M}$ is an $X$-computable copy of $\mc{M}$.
			\end{proof}
			
			These claims complete the proof of the lemma. By Claim \ref{cl:one} there is no $\mathbf{0}^{(n)}$-computable copy of $\mc{M}$. If $X \nleq \mathbf{0}^{(2n)}$ then either $X$ or $\overline{X}$ is not $\Sigma^0_{2n+1}$ and so by Claim \ref{cl:two} there is an $X$-computable copy of $\mc{M}$. This is uniform in $n$. (Recall that if $X$ is non-arithmetic, then $X$ is not $\Sigma^0_{2n+1}$, and so this is uniform in $X$ as well.)
		\end{proof}
		
		We have now proved the lemma for all $n$. For each $n$ there is a structure $\mc{M}_n$ which has no $\mathbf{0}^{(n)}$-computable copy, but for each set $X$ which is not $\Sigma^0_{2n+1}$, we can build (uniformly in $X$) an $X$-computable copy of $\mc{M}_n$. Putting all of these structures together as described just after the statement of the lemma, we obtain a structure whose degree spectrum is exactly the non-arithmetic degrees. 
	\end{proof}

	\bibliography{References}
	\bibliographystyle{alpha}
	
\end{document}